\documentclass[11pt,twoside, leqno]{article}

\usepackage{amssymb,amsmath,amsfonts,amsthm,color,mathrsfs}
\usepackage[Symbol]{upgreek}
\usepackage{txfonts}
\usepackage[nottoc,notlot,notlof]{tocbibind}
\usepackage[active]{srcltx}
\usepackage{hyperref}

\usepackage{graphicx}
\usepackage{epsfig}

\usepackage{txfonts}
\allowdisplaybreaks
\def\rr{{\mathbb R}}

\def\cn{{\mathbb N}}

\def\D{{\mathscr D}}

\def\ccc{{\mathscr C}}

\def\fz{\infty}

\def\dist{{\mathop\mathrm {\,dist\,}}}
\def\diam{{\mathop\mathrm {\,diam\,}}}
\def\loc{{\mathop\mathrm{\,loc\,}}}

\def\ez{\varepsilon}

\def\E{\mathscr{E}}

\def\r{\right}
\def\lf{\left}

\newtheorem{thm}{Theorem}[section]
\newtheorem{lem}[thm]{Lemma}
\newtheorem{prop}[thm]{Proposition}
\newtheorem{cor}[thm]{Corollary}

\newtheorem{quest}[thm]{Question}

\newtheorem{defn}[thm]{Definition}
\newtheorem{rem}[thm]{Remark}

\numberwithin{equation}{section}

\begin{document}
\arraycolsep=1pt
\author{Renjin Jiang}
\title{{\bf Riesz transform via heat kernel and harmonic functions on non-compact manifolds}
 \footnotetext{\hspace{-0.35cm} 2010 {\it Mathematics
Subject Classification}. Primary  58J35; 42B20; Secondary  58J05; 35B65; 35K05.
\endgraf{
{\it Key words and phrases: Riesz transform, harmonic functions, heat kernel, Poincar\'e inequality, doubling measure}
\endgraf}}
\date{}}
\maketitle

\begin{center}
\begin{minipage}{11.5cm}\small
{\noindent{\bf Abstract}.  Let $M$ be a complete  non-compact manifold satisfying the volume doubling  condition, with doubling index $N$ and reverse doubling index $n$, $n\le N$, both for large balls.
Assume a Gaussian upper bound for the heat kernel,  and
an $L^2$-Poincar\'e inequality outside a compact set.

If $2<n$, then we  show that for $p\in (2,n)$, $(R_p)$: $L^p$-boundedness of the Riesz transform,
$(G_p)$: $L^p$-boundedness  of the gradient of the  heat semigroup, and $(RH_p)$: reverse $L^p$-H\"older inequality for the gradient of
harmonic functions, are equivalent to each other.
Our characterization implies that for $p\in (2,n)$, $(R_p)$ has
an open ended property and is stable under gluing operations.
This substantially extends the well known equivalence of $(R_p)$ and $(G_p)$  from \cite{acdh} to more general settings, and is optimal in the sense that $(R_p)$ does not hold  for any $p\ge n>2$ on manifolds having at least two Euclidean ends of dimension $n$.

For $p\in (\max\{N,2\},\infty)$, the fact that $(R_p)$, $(G_p)$ and $(RH_p)$ are equivalent essentially follows from \cite{cjks16};  moreover, if $M$ is non-parabolic, then any of these conditions implies that $M$ has only one end.

For the proof, we develop a new criteria for boundedness of the Riesz transform,
which was nontrivially adapted from \cite{acdh}, and make an essential application
of results from \cite{cjks16}. Our result allows extensions to non-smooth settings.


}\end{minipage}
\end{center}
\vspace{0.2cm}
\tableofcontents

\section{Introduction}

\subsection{Background and motivations}

\hskip\parindent Let $M$ be a complete, connected and  non-compact Riemannian manifold. Denote by $d$ the geodesic distance, by $\mu$ the
Riemannian measure, and by $\mathcal{L}$  the non-negative Laplace-Beltrami operator   on $M$.  Let $\{e^{-t\mathcal{L}}\}_{t>0}$ be the heat semigroup.
 The inverse of the square root of $\mathcal{L}$ is given by
  $$\mathcal{L}^{-1/2}=\frac{\sqrt{\pi}}{2}\int_0^\infty e^{-s\mathcal{L}}\frac{\,ds}{\sqrt s}.$$
Denote by $\nabla$ the Riemannian gradient.

The study of the Riesz transform $\nabla \mathcal{L}^{-1/2}$ is one of the central topics of analysis on manifolds.
 Strichartz in 1983 \cite{str83}, and then Bakry in 1987 \cite{bak2}, provided  sufficient conditions
 on non-compact manifolds such that the Riesz transform
 is bounded for all $1<p<\infty$ (see Chen \cite{ch92} for the case $p=1$). Since then, many sufficient, or even in some cases necessary and sufficient, conditions
 for the boundedness of the Riesz transform have been provided;
see  for instance \cite{Al92,ac05,acdh,ca07,cch06,CMO15,cd99,cd03,lj91,lx10}. Let us review some related results.
Since the boundedness of Riesz transform on compact manifolds is not an issue (cf. \cite{str83}), we will only consider non-compact cases.

For each $p\in (1,\infty)$, we say that  $(R_p)$ holds,
 if the Riesz transform $|\nabla \mathcal{L}^{-1/2}|$ is bounded on $L^p(M)$.
Notice that $(R_2)$ holds automatically which can be seen by  integration by parts.
In the metric measure space $(M,d,\mu)$, denote by $B(x,r)$ the open ball with centre $x\in M$ and radius $r>0$ and by $V(x,r)$ its volume
$\mu(B(x,r))$.
One says that $M$ satisfies the volume doubling property (in short is  doubling)  if there exists  a constant  $C_{D}>1$ such that
$$ V(x,2r)\le C_{D}V(x,r), \leqno(D)$$
 for all $r>0$  and $x\in M$.
The heat semigroup has a smooth positive and symmetric kernel $p_t(x,y)$,  meaning that
$$e^{-t\mathcal{L}}f(x)=\int_Mp_t(x,y)f(y)\,d\mu(y)$$
for suitable  functions $f$.
One  says that the heat kernel satisfies a Gaussian upper  bound if
there exist $C,c>0$ such that for all $t>0$ and $x,y\in M$,
$$p_t(x,y)\le
  \frac{C}{V(x,{\sqrt t})}\exp\lf\{-\frac{d^2(x,y)}{ct}\r\}.\leqno(UE)
$$
Coulhon and Duong \cite{cd99} showed that the doubling condition together with a Gaussian upper bound of heat kernel is sufficient
for  $(R_p)$ for all $p\in (1,2)$.
Recently, Chen et al. \cite{CCFR} showed, a bit surprisingly, that
a  sub-Gaussian  upper bound of the  heat kernel could replace the Gaussian upper bound in the above result;
see \cite{lz17} for further developments.

The case $p>2$ is more difficult. Notice that if $(R_p)$ holds,
then it follows from the analytic property of the heat semigroup that
$$\||\nabla e^{-t\mathcal{L}}|\|_{p\to p}\le \||\nabla \mathcal{L}^{-1/2}\mathcal{L}^{1/2}e^{-t\mathcal{L}}|\|_{p\to p}\le  \frac{C}{\sqrt t}\leqno(G_p)$$
for all $t>0$; see \cite{st70} or \cite{acdh}. Above and in what follows, we use the notation $\|\cdot\|_p$ to denote the $L^p$ norm over $M$,
and the notation $\|\cdot\|_{p\to p}$ for the operator norm from $L^p$ to $L^p$,  for any $p\in [1,\infty]$.
A natural and longstanding question is as following.

\begin{quest}\label{question-gp}
Let $p\in (2,\infty)$. Does $(G_p)$ imply $(R_p)$?
\end{quest}

Auscher, Coulhon, Duong and Hofmann  in 2004 \cite{acdh} established a remarkable result, which shows that, under $(D)$ and a scale-invariant $L^2$-Poincar\'e inequality, $(G_{p_0})$ implies $(R_p)$
 for all $p\in (2,p_0)$, where $p_0\in (2,\infty]$.
The scale invariant $L^2$-Poincar\'e inequality means that there exists $C>0$ such that
for every ball $B$ and each $f\in C^1(\bar B)$, it holds
$$
\fint_{B}|f-f_B|^2\,d\mu\le
Cr_B^2\fint_{B}|\nabla f|^2\,d\mu,\leqno(P_{2})$$
where $f_B$ denotes the average of the integral of $f$ on $B$.
Notice that $(D)$ together with $(P_2)$ is equivalent to a two-sided Gaussian bound for the heat kernel;
see \cite{gri92,sal2}. By recent results from \cite{bf15,cjks16},
   one finally sees that $(G_p)\Longleftrightarrow(R_p)$ for each $p\in (2,\infty)$, under $(D)$ and $(P_2)$.

However, Question \ref{question-gp} in generality is still open; see \cite[Subsection 1.4]{acdh} and also \cite{bf15}. The  requirement of $(P_2)$ is  not necessary by looking at a manifold obtained by gluing two Euclidean
ends through a compact manifold smoothly; see \cite{cch06,cd99,gri-sal09}.
Here and below,  an {\em end} means,  an unbounded component of  a complete non-compact manifold $M$ outside a
compact subset $M_0$.

In \cite{cch06}, Carron, Coulhon and Hassell
showed that the Riesz transform  is $L^p$-bounded  for $2<p<n$, $n\ge 3$, if  $M$ is an $n$-dimensional manifold
with  a finite number of  Euclidean ends; the result has been further generalized to manifolds with conic ends
 by Guillarmou and Hassell \cite{gh08}, and by Carron \cite{ca16} to manifolds with quadratic Ricci curvature decay, i.e., for a fixed $x_M\in M$ and $C_M\ge 0$, it holds
$$Ric_M(x)\ge -\frac{C_M}{[d(x,x_M)+1]^2}.\leqno(QD)$$
Moreover, in \cite{cch06}, it has been showed that if $M$ has at least two ends, then the Riesz transform is not $L^p$-bounded
for any $p\ge n$. Indeed, by using $L^p$-cohomology, the following non-trivial result was proved in \cite{cch06}.

\begin{thm}[\cite{cch06}]\label{negative}
Suppose that $M$ has Ricci curvature bounded from below,  and  for some $N>2$
$ V(x,r)\lesssim r^N$, for all $x\in M$ and $r\ge 1. $
If there exists $C>0$ such that for any $f\in C^\infty_c(M)$ it holds
$$\|f\|_{\frac{2N}{N-2}}\le C\||\nabla f|\|_2.\leqno(S_{\frac{2N}{N-2},2}),$$
and $M$ has at least two ends, then the Riesz transform is not bounded on $L^p(M)$ for any $p\ge N$.
\end{thm}

Notice that the Sobolev inequality $(S_{\frac{2N}{N-2},2})$ together with $ V(x,r)\lesssim r^N$ implies $(UE)$ (cf. \cite{gri92,gri-sal09}), conversely $(UE)$ only implies a local Sobolev inequality (cf. \cite{bcs15,gri92,gri-sal09}). In particular, under $(UE)$, $(S_{\frac{2N}{N-2},2})$ may not hold; see \cite{var88}.

The above result has been further refined by Carron \cite[Theorem C]{ca16}. Notice that, in particular, in the above theorem and Carron's theorem,
the ends are not necessarily Euclidean or conic. In view of this, in \cite{cch06},
several questions, regarding relaxing the requirement that ends are Euclidean, had been proposed;
see following Question \ref{question-gluing}, Question \ref{question-Li} and Question \ref{question-Lie}.

In this paper, we provide a solution to Question \ref{question-gp} by relaxing the requirement of $(P_2)$,
but only for $p$ in the intervals $(2,n)$ and $(\max\{2, N\}, \infty)$; see Theorem \ref{main-n} and Theorem \ref{main-one} below.
As an application, we obtain stability under gluing operation and open ended property for the Riesz transform on manifolds with general ends.
Notice that the case $p\in (1,2)$ was well understood by \cite{cd99,CCFR}, as we recalled above.
We will only consider the case $p>2$ in this work.

Throughout the paper, we assume that $M$ is a non-compact, connected and complete manifold, that satisfies the
doubling condition $(D)$. We shall simply recognize  $M$ as the union of a compact set $M_0$ and one or more but finitely many ends $\{E_i\}_i$.
We fix a point $x_M\in M_0$ and assume without loss of generality that $\mathrm{diam}(M_0)=1$.

The doubling condition $(D)$ together with connectedness implies that there exist $0<\upsilon\le \Upsilon<\infty$ such that for any $x\in M$ and all $0<r<R<\infty$ it holds
\begin{equation}\label{reverse-doubling}
\left(\frac{R}{r}\right)^{\upsilon}\lesssim \frac{V(x,R)}{V(x,r)}\lesssim \left(\frac{R}{r}\right)^{\Upsilon};
\end{equation}
see for instance \cite[p. 213, Remark 8.1.15]{hkst}.
This further implies that there exists $0<N<\infty$ such that
$$\frac{V(x,R)}{V(x,r)}\lesssim \left(\frac{R}{r}\right)^N,\ \forall\,x\in M \ \& \ \forall \,1<r<R<\infty,
\leqno(D_{N})$$
and there exists $0<n\le N$ such that
$$\left(\frac{R}{r}\right)^n\lesssim \frac{V(x_M,R)}{V(x_M,r)},\ \forall\,1<r<R<\infty, \leqno(RD_{n})$$
where $x_M\in M_0$  is a fixed point.
In what follows, we call $n$ the lower dimension, and $N$ the upper dimension, of $M$. Moreover, we simply use $(D_{N})$
to indicate that $\mu$ is a doubling measure with $N$ being the upper dimension.

\begin{rem}\rm
(i) It holds obviously $\upsilon\le n\le N\le \Upsilon$. The examples of cocompact covering Riemannian manifolds with  polynomial growth deck transformation group and Lie groups of polynomial growth show it may happen that
$\upsilon<n$ and $N<\Upsilon$; see \cite{Al92,dng04a,hs09,var88} for instance.

(ii) Notice that we only need $(RD_{n})$ for a fixed point $x_M\in M_0$ and $R>r>1$.
Take weighted lines $(\mathbb{R},\,(1+|x|)^\alpha\,dx)$, $\alpha>0$, for example.
A small calculation shows that $(D_{\alpha+1})$ and $(RD_{\alpha+1})$ hold, but \eqref{reverse-doubling}
holds with $\upsilon=1$ and $\Upsilon=\alpha+1$; see \cite{hs09} and also \cite{ca16}.
Moreover, by using the doubling property and the fact $M_0$ is compact, one sees that $(RD_{n})$ holds if and only
 it holds for each $o\in M_0$ and all $1<r<R<\infty$ that
$\left({R}/{r}\right)^n\lesssim {V(o,R)}/{V(o,r)}.$

(iii) In many cases, such as manifolds with conic ends, or with ends like cocompact covering Riemannian manifolds with  polynomial growth deck transformation group or Lie groups of polynomial growth, one has $n=N$.
\end{rem}

By Theorem \ref{negative} and \cite[Theorem C]{ca16}, we already see that the (homogenous) dimension plays a key role in the Riesz transform.
It is then naturally to split the case $p>2$ into two categories: $p$ less than the dimension and $p$ bigger than the dimension.
We will  provide necessary and sufficient conditions for boundedness of the Riesz transform in both cases.

We first consider $p>2$ that is smaller than the lower dimension $n$, which means that $n>2$ and  the ends are non-parabolic;
see Subsection \ref{sub-largep} for the definition and \cite{ca16,lt92} for more materials.
Recall that if a manifold has two  Euclidean or conic ends of dimension two, then the Riesz transform is not
$L^p$-bounded for any $p>2$ by \cite{ca16,cd99}.

For $p>2$ that is bigger than the upper dimension $N$, we will consider manifolds with general ends (including small ones).
Notice that in this case boundedness of the Riesz transform will imply that the manifold can have only {\em one} end,
if $M$ is non-parabolic; see Theorem \ref{main-one} below.

\subsection{Necessary and sufficient conditions for small $p$  }\label{sub-smallp}
\hskip\parindent  In this part, we provide a necessary and sufficient condition for $L^p$-boundedness of the Riesz transform for small $p$,
i.e., $p$ less than the lower dimension.  Our approach depends heavily on recent developments on the relation of regularities of
harmonic functions and heat kernels from \cite{cjks16,ji15,jky14}, and is a nontrivial adaption of  the criteria for
the boundedness of the Riesz transform established in \cite{acdh} (see also \cite{Au07}) to our settings.

\begin{defn}[Poincar\'e inequality]
We say that a Poincar\'e inequality holds on ends ($(P_2^E)$, for short) of  $M$, if there exists $C>0$ such that
for any ball $B$ with $2B\cap M_0=\emptyset$,  and each $f\in C^1(\bar B)$,
$$
\fint_{B}|f-f_B|^2\,d\mu\le
Cr_B^2\fint_{B}|\nabla f|^2\,d\mu.\leqno(P_{2}^E)$$
\end{defn}

 Our first main result provides a solution to Question \ref{question-gp} for $p\in (2,n)$.
Notice that, under the doubling condition, our assumptions $(UE)$ and $(P^E_2)$ below are much weaker than $(P_2)$.
For example, $(UE)$ and $(P^E_2)$ hold on a manifold obtained by gluing two copies of Eulcidean space $\rr^n$ together, $n\ge 2$,
while $(P_2)$ does not hold; see \cite{ca07,ca16,cch06,cd99,gh08} for instance.
\begin{thm}\label{main-n}
Assume that $(D_{N})$ and $(RD_n)$ hold on $M$  with $2<n\le N<\infty$.
Suppose that $(UE)$ and $(P^E_2)$ hold.  Let $p\in (2,n)$. Then the following statements are equivalent.

(i) $(R_p)$ holds;

(ii) $(G_p)$ holds;

(iii) $(RH_p)$ holds, where $(RH_p)$ means that there exists $C>0$ such that for any ball $B$ with radius $r_B$ and any harmonic function $u$ on $3B$, it holds
$$\left(\fint_B|\nabla u|^p\,d\mu\right)^{1/p}\le \frac{C}{r_B}\fint_{2B}|u|\,d\mu. \leqno(RH_p)$$
\end{thm}

\begin{rem}\rm
(i) Under assumptions of the theorem, $(R_p)$ holds for all $p\in (1,2]$ from Coulhon-Duong \cite{cd99},
and does not hold for any $p\ge N$, if the manifold has
at least two ends, by Carron \cite[Theorem C]{ca16}.

(ii) The condition $(RH_p)$ is different from the true reverse H\"older inequalities used in \cite{ac05,shz05}.
Our formulation is natural since in case of manifolds with two Euclidean/conic ends,  the true reverse H\"older inequalities
fail for any $p>2$, but $(RH_p)$ holds for $p\in (2,n)$; see \cite[Section 7]{cjks16}.

(iii) The equivalence $(G_p)\Longleftrightarrow (RH_p)$ was proved in \cite{cjks16} under a local Poincar\'e inequality
$(P_{2,\loc})$ instead of $(P^E_2)$. Notice that $(P^E_2)$ implies $(P_{2,\loc})$; see Lemma \ref{pend-ploc} below.
\end{rem}

In view of Theorem \ref{negative} and \cite[Theorem C]{ca16}, the above result is rather optimal if the manifold has at least two ends. It is worth to note that our method are completely different from those from \cite{ca07,ca16,cch06,gh08},
in particular, our assumptions $(D)$, $(UE)$ and $(P^E_2)$ all are stable
under quasi-isometries. As a consequence, our results work with the Laplace-Beltrami operator replaced
by any {\em uniformly elliptic operator} of divergence form, and more generally, work on Dirichlet metric measure spaces; see
Section \ref{extension-mms}.

 The condition $(P^E_{2})$ is satisfied on an end, if the Ricci curvature has quadratic decay $(QD)$ (see Buser \cite{bu82} or Theorem \ref{asy-poin}), or the end is quasi-isometric to one of the following manifold removing a compact set: a co-compact covering manifold with  polynomial growth deck transformation group, Lie group of polynomial growth as well as conic manifold; see \cite{Al92,chy,cl04,dng04a,lh99,var88,ya75} for instance.

The condition $(UE)$ is a global condition and seems to be more restrictive. However, recent results of  Grigor'yan and Saloff-Coste \cite{gri-sal09,gri-sal16} shed some light on this point.  In particular, by  \cite{gri-sal16} one sees that  if each end $E_i$ is isometric to $\widetilde{M}_i\setminus K_i$, where $\widetilde{M}_i$  is a complete manifold  satisfying $(UE)$ and $K_i$ is a compact set, then the manifold $M$ satisfies $(UE)$; see the final section.

We next provide some further necessary and sufficient conditions for the boundedness of the Riesz transform.

\begin{defn}  Let $p\in (2,\infty]$.
We say that  the reverse
$L^p$-H\"older inequality for gradients of harmonic functions holds on ends of   $M$ (for short, $(RH^E_{p})$), if  there exists $C>0$
 such that for each ball $B$ with $3B\cap M_0=\emptyset$, and each  harmonic function $u$ on $3B$, it holds
$$\left(\fint_{B}|\nabla u|^p\,d\mu\right)^{1/p}\le \frac {C}{r_B} \fint_{2B}|u|\,d\mu.
\leqno (RH^E_{p})$$
\end{defn}

The observation below is that, $(RH_p)$ is stable under gluing operation, if $p<n$; see
Lemma \ref{stable-rhp} and Lemma \ref{har-pend-ploc} below.
\begin{thm}\label{main}
Assume that $(D_{N})$ and $(RD_n)$ hold on $M$  with $2<n\le N<\infty$.
Suppose that $(UE)$ and $(P^E_2)$ hold. Let $p\in (2,n)$. Then $(R_p)$ holds on $M$, if and only if,   $(RH^E_p)$ holds.
\end{thm}
The advantage  is that  $(RH^E_p)$ is a condition much easier to verify than $(G_p)$.
An immediate consequence of the above result is that compact metric perturbation does not affect $(R_p)$, if $p<n$.
We also note that,  the above result implies the stability of $(R_p)$ ($p<n$) under gluing operations, see
Theorem \ref{stability-gluing-main} and Corollary \ref{stability-gluing} below.

An open-ended property of the Riesz transform follows from the above result.
\begin{cor}\label{cor-open}
Assume that $(D_{N})$ and $(RD_n)$ hold on $M$  with $2<n\le N<\infty$.
Suppose that $(UE)$ and $(P^E_2)$ hold. Let $p\in (2,n)$. If $(R_p)$ holds, then there exists $\epsilon>0$ such that $p+\epsilon<n$
and $(R_{p+\epsilon})$ holds.
\end{cor}

Further, $(RH^E_\infty)$ and $(P^E_2)$ hold if the Ricci curvature has quadratic decay.
\begin{cor}\label{main-1}
Assume that $(D_{N})$ and $(RD_n)$ hold on $M$  with $2<n\le N<\infty$.  If $(UE)$ holds, and there exists $C_M>0$ such that for each $x\in M$,
$$Ric_M(x)\ge -\frac{C_M}{[d(x,x_M)+1]^2},\leqno(QD)$$
then $(R_p)$ holds for $p\in (1,n)$.
\end{cor}
Carron \cite[Thoerem A]{ca16} established that if a manifold satisfies a volume comparison condition (VC), the (RCE) condition (relatively connected to an end), $(QD)$ and $(RD_n)$, then $(R_p)$ holds for $p\in (1,n)$. See also Devyver \cite[Theorem 5]{de14}
for a related result. Note that Carron's assumptions imply $(D)$
and $(UE)$; see \cite[Section 2]{ca16}. However, after a careful reading of \cite[Section 3 and Section 4]{ca16} we find that
Carron's proof indeed works under our assumptions in the above Corollary.
The approach \cite{ca16} used Li-Yau's Harnack inequality (cf. \cite{ly86}) to deduce point-wise behaviors of the
Riesz kernel, which depends on the smooth structure.
Our approach (after applying Theorem \ref{main-n} and Theorem \ref{main}) needs to verify the regularity of harmonic functions
on the ends, and can be applied to deal with general uniformly elliptic operators on such manifolds (see Section 5).

\subsection{Necessary and sufficient conditions for large $p$}\label{sub-largep}
\hskip\parindent Theorem \ref{main-n} seems to be rather optimal if the manifold has at least two ends, however,
it is less satisfied if the manifold  has only one end, where in general the Riesz transform may be bounded on $L^p(M)$
for some $p>N$; see \cite{ca16,gh08,lh99} for instance. We next provide
a necessary and sufficient condition for $p>N$ under the same requirements as Theorem \ref{main-n} except
that we do not need the reverse doubling condition, which however holds automatically. The result in this part essentially follows
from \cite{cjks16}.
We shall denote   $\max\{A,\,B\}$
by  $A\vee B$.

Let us recall some notation regarding parabolic and hyperbolic manifolds; see \cite{ca16} for instance.
Let $p\in (1,\infty)$. For a bounded open set $O\subset M$, define its $p$-capacity by
$$\mathrm{Cap}_p(O):=\inf\left\{\int_{M}|\nabla \psi|^p\,d\mu,\, \psi\in C^\infty_c(M), \psi\ge 1 \, \mbox{on}\, O\right\}.$$
We say that $M$  is $p$-hyperbolic if the $p$-capacity of some (equivalently, any) bounded open subsets is positive.
A non-$p$-hyperbolic manifold is called $p$-parabolic. A $2$-hyperbolic manifold is called non-parabolic.

\begin{thm}\label{main-one}
Assume that $(D_{N})$ holds on $M$  with $0<N<\infty$, and that $(UE)$ and $(P^E_2)$ hold.  Let $p\in (N\vee 2,\infty)$. Then the following statements are equivalent.

(i) $(R_p)$ holds;

(ii) $(RH_p)$ holds;

(iii) $(G_p)$ holds.

Moreover, if $M$ is non-parabolic, then any of the three conditions implies that $M$ can have only one end.
\end{thm}
Note that we did not assume $(P_2)$ above, however $(P_2)$ follows as a consequence of the proof; see Remark \ref{remark-p2}.

For the proof we will show that the validity
$(P^E_2)$ guarantees scale-invariant Poincar\'e inequalities $(P_p)$ for any $p>N\vee2$  (see Theorem \ref{n-poin}).  The validity of these Poincar\'e inequalities
allows us to use \cite[Theorem 1.9]{cjks16}, and then \cite[Theorem C]{ca16} to conclude the theorem.

We have the following unboundedness of the Riesz transform as an application of the above result.
\begin{cor}\label{main-p2-negative}
Assume that $(D_{N})$ holds on $M$  with $0<N<\infty$, and that $(UE)$ and $(P^E_2)$ hold.  If there exists a non-constant harmonic function $u$ on $M$ with the growth
$$u(x)=\mathcal{O}(d(x,o)^\alpha)\ \mbox{as}\, d(x,o)\to\infty$$
for some $\alpha\in [0,1)$ and a fixed $o\in M$,
then $(R_p)$ does not hold for any  $p>N\vee 2$ satisfying $p(1-\alpha)\ge N$.
\end{cor}
Using the Poincar\'e inequality $(P_p)$ for any $p>N\vee 2$ established in Theorem \ref{n-poin}
together with  Theorem \ref{main-one} allows us to conclude the claim via arguing by contradiction.
We refer the reader to \cite{ca16,lt92} for more on existence and non-existence of non-constant harmonic functions of
sublinear growth.

\subsection{Applications and comments}
\hskip\parindent As applications of our main results, in this part, we address the questions of stability of boundedness
of the Riesz transform under gluing operations and make some final comments on our result.

 The following question was asked in \cite{cch06}.
\begin{quest}[Part of Open Problem 8.2 \cite{cch06}]\label{question-gluing}
Under which conditions is boundedness of the Riesz transform on $L^p$ stable under the
gluing operation on manifolds?
\end{quest}
We refer the reader to \cite[Section 3]{gri-sal16} and also \cite{gri-sal09} for a detailed description of the gluing operation. Here we only need to know that the gluing operation is smooth, and only changes structure and metric in a compact set.
As shown by Theorem \ref{negative}, \cite[Theorem C]{ca16} and Theorem \ref{main-one}, the $L^p$-boundedness of the Riesz transform
is not stable under the gluing operations if $p$ is not less than the dimension $N$ and bigger than two.
Previously, Carron \cite{ca07} and Devyver \cite{de15} had addressed this question under the requirement of lower Ricci curvature bound and Sobolev inequalities;
see also \cite{ca16} for a description of Devyver's result.

Our Theorem \ref{main} provides a solution to the above question in a different manner than \cite{ca07,de15}.
\begin{thm}\label{stability-gluing-main}
Let $k\ge 2$. Suppose that for each $1\le i\le k$, $M_i$ is a complete non-compact manifold where $(D)$, $(UE)$ and $(P_2^E)$ hold.
Assume that the gluing manifold $M:=M_1\#\cdots \#M_k$  satisfies $(D_{N})$ and $(RD_n)$ for some $2<n\le N<\infty$.
Then if for some $p\in (2,n)$, $(R_p)$ holds on each $M_i$, $(R_p)$ holds on $M$.
\end{thm}
It is worth to note that, our assumptions, $(UE)$ and $(P_2^E)$, are stable under gluing operations.
Indeed, under gluing operations, it is straight to see that $(P_2^E)$ is stable, on the other hand, the stability of $(UE)$ follows from \cite{gri-sal09,gri-sal16} (see Theorem \ref{upper-heat}).

Since $(P_2)$ implies $(UE)$ and $(P_2^E)$, we obtain the following corollary.
\begin{cor}\label{stability-gluing}
Let $k\ge 2$. Suppose that for each $1\le i\le k$, $M_i$ is a complete non-compact manifold where $(D)$ and $(P_2)$ hold.
Assume that the gluing manifold $M:=M_1\#\cdots \#M_k$  satisfies $(D_{N})$ and $(RD_n)$ for some $2<n\le N<\infty$.

(i) There exists $\epsilon>0$ such that $2+\epsilon<n$ and $(R_{2+\epsilon})$ holds.

(ii) If for some $p\in (2,n)$, $(R_p)$ holds on each $M_i$, then $(R_p)$ holds on $M$.
\end{cor}



In \cite{cch06},  some open  questions regarding manifolds with conic ends or ends isometric to
simply connected nilpotent Lie groups at infinity  were also proposed.
These two questions were solved by Guillarmou and Hassell \cite{gh08} and Carron \cite{ca07}, respectively;
  see also Carron \cite{ca16}.
Our results also provide a new proof to the two questions; see Section  \ref{sec-ue}.

Finally, let us make some comments. Notice that our main results, Theorem \ref{main-n}, Theorem \ref{main} and Theorem \ref{main-one}, together with \cite[Theorem C]{ca16} and \cite{cd99},
give a more or less satisfactory  solution for the Riesz transform on manifolds with ends, for the two cases: (i) $1<p<\max\{2,n\}$,
 (ii) $p>N\vee 2$. Recall that the case $p=2$ is trivially true.

Note that for manifolds with ends like  Euclidean ends, conic ends, or ends at infinity isometric to Lie groups of polynomial growth or cocompact covering Riemannian manifolds with  polynomial growth deck transformation group,
it holds that $n=N$.  It turns out that on these settings,  $(R_p)$ is stable under gluing operation for $p<n=N$ by Theorem \ref{stability-gluing-main}, and not stable for $p\ge N$ by \cite[Theorem C]{ca16}. It is then somehow
not restrictive to assume $(P_2)$ for $p\ge N$ or may necessary to have $(P_2)$, under which $(R_p)$ for $p>2$
is well understood  by \cite{acdh}, see also \cite{cjks16} and Theorem \ref{main-one}.

However, for manifolds where one only has $n<N$, the case $p\in (2,\infty) \cap [n,N]$ is still unclear,
and certainly deserves further study.
\begin{quest}\label{open-future}
Let $M$ be a complete non-compact manifold, which satisfies $(D_{N})$ and $(RD_n)$ for some $0<n\le N<\infty$ and $N>2$. Suppose that $(UE)$ and $(P^E_2)$ hold.
Then is $(R_p)$ equivalent to $(G_p)$ or $(RH_p)$ for $p\in (2,\infty) \cap [n,N]$?
\end{quest}
For each  $p\in (2,\infty)$, it was known from \cite{cjks16} that $(G_p)\Longleftrightarrow (RH_p)$, and it holds automatically that
$(R_p)\Longrightarrow (G_p)$ (cf. \cite{acdh}). So the only question left is, does $(G_p)$ or $(RH_p)$ imply $(R_p)$ for $p\in (2,\infty) \cap [n,N]$?

\subsection{Structure of the paper}
\hskip\parindent
 The paper is organized as follows. In Section 2, we provide various versions of Poincar\'e inequalities for later use.
In Section 3, we study the Riesz transform for $p$ less than the lower dimension, while
in Section 4, we study the case $p$ bigger than the upper dimension. In Section 5, we provide some extensions
of the main results to non-smooth settings.
In the final section, we shall discuss the validity of $(UE)$,
and provide examples that our results can be applied to, in particular, we give the proof of Theorem \ref{stability-gluing-main} and Corollary \ref{stability-gluing}.

Throughout the work, we denote by $C,c$ positive constants which are independent of the
main parameters, but which may vary from line to line. For a ball $B$,
unless otherwisely specified, we denote its radius and center by $r_B$ and $x_B$, respectively.

\section{Poincar\'e inequality}\label{sc-sp}
\hskip\parindent In this section, we shall provide various versions of Poincar\'e inequalities for later use.
%

\begin{defn}[Hardy-Littlewood maximal function]
For any locally integrable function $f$ on $M$, its Hardy-Littlewood maximal function is defined as
$$\mathcal{M}f(x):=\sup_{B:\, x\in B}\fint_B|f|\,d\mu,$$
where $B$ is any ball that contains $x$.
For $p>1$, we define the $p$-Hardy-Littlewood maximal function as
$$\mathcal{M}_pf(x):=\sup_{B:\, x\in B}\left(\fint_B|f|^p\,d\mu\right)^{1/p}.$$
\end{defn}

We say that $M$ supports a local $L^2$-Poincar\'e inequality (for short, $(P_{2,\loc})$), if for all $r_0>0$ there exists
$C_P(r_0)>0$ such that, for every ball $B$ with $r_B<r_0$ and each $f\in C^1(\bar B)$,
$$
\fint_{B}|f-f_B|^2\,d\mu\le
C_P(r_0)r_B^2\fint_{B}|\nabla f|^2\,d\mu.\leqno(P_{2,\loc})$$

\begin{lem}\label{pend-ploc}
Assume that $(P^E_2)$ holds on $M$, then $(P_{2,\loc})$ holds on $M$.
\end{lem}
\begin{proof}
For any $r_0>0$,  the Ricci curvature  on the set $\{x\in M:\, \dist (x,M_0)<3r_0\}$ is bounded below by a constant $K(r_0)$ depending on $r_0$.  Therefore, by Buser \cite{bu82} (see also \cite{hak}), there exists $C_{P}(r_0)$ such that
 for every ball $B=B(x,r)$ with $r<r_0$ and $\dist (x,M_0)<2r_0$,   and each $f\in C^1(\bar B)$, it holds
$$
\fint_{B}|f-f_B|^2\,d\mu\le
C_P(r_0)r^2\fint_{B}|\nabla f|^2\,d\mu.\leqno(P_{2,\loc})$$
On the other hand, by  $(P^E_2)$, one sees that
 there exists $C$ such that   for any ball $B(x,r)$ with center $x\notin \{y\in M:\, \dist (y,M_0)<2r_0\}$ and $r<r_0$, it holds
  for each $f\in C^1(\bar B)$ that
$$
\fint_{B}|f-f_B|^2\,d\mu\le
Cr^2\fint_{B}|\nabla f|^2\,d\mu,$$
as desired.
\end{proof}

For a real number $\gamma>0$ we denote
by $[\log_2 \gamma]$ the biggest integer not bigger than $\log_2 \gamma$.
\begin{thm}\label{n-poin}
Assume that $(D_{N})$ holds on $M$ with $0<N<\infty$.
If  $(P^E_2)$ holds on $M$, then for any $p>N\vee 2$ there is a Poincar\'e inequality $(P_p)$, i.e., there exists $C>0$ such that for any ball $B$
and any $f\in C^1(\bar{B})$ it holds
$$\fint_{B}|f-f_B|\,d\mu\le Cr_B\left(\fint_{B}|\nabla f|^p\,d\mu\right)^{1/p}.\leqno(P_p)$$
\end{thm}
\begin{proof}
Since $(M,d)$ is a geodesic space, by Haj\l asz-Koskela \cite[Section 9]{hak}, it suffices to prove the following weaker version, i.e., for
$f\in C^1(\overline{8B})$,
$$\fint_{B}|f-f_B|\,d\mu\le Cr_B\left(\fint_{8B}|\nabla f|^p\,d\mu\right)^{1/p}.\leqno(\widetilde{P_p})$$

By Lemma \ref{pend-ploc}, a local Poincar\'e inequality $(P_{2,\loc})$ holds.
If $r_B\le 100$, then the required estimate $(\widetilde{P_p})$ follows from $(P_{2,\loc})$.

Assume now $r_B>100.$
If $2B\cap M_0=\emptyset$, then $(P_p)$ and hence $(\widetilde{P_p})$ follows from $(P^E_2)$.

Suppose $2B\cap M_0\neq \emptyset$. Let $f\in C^1(\overline{8B})$ and write
$$\fint_{B}|f-f_B|\,d\mu\le  \fint_{B}\fint_{B} |f(x)-f(y)|\,d\mu(x)\,d\mu(y). $$
{\bf Claim}: For each $q\in (N\vee 2,\infty)$, there is a constant $C>0$ such that for all $x,y\in B$  it holds
$$|f(x)-f(y)|\le Cr_B \left[\mathcal{M}_q\left(|\nabla f|\chi_{8B}\right)(x)+\mathcal{M}_q\left(|\nabla f|\chi_{8B}\right)(y)
\right].$$
If the claim holds,  then by taking $q\in (N\vee 2,p)$, we conclude that
\begin{eqnarray*}
\fint_{B}|f-f_B|\,d\mu&&\le Cr_B\fint_{B}\fint_{B}\left[\mathcal{M}_q\left(|\nabla f|\chi_{8B}\right)(x)+\mathcal{M}_q\left(|\nabla f|\chi_{8B}\right)(y)
\right]\,d\mu(x)\,d\mu(y)\\
&&\le Cr_B\left(\fint_{B}\left[\mathcal{M}_q\left(|\nabla f|\chi_{8B}\right)(x)\right]^p\,d\mu(x)\right)^{1/p} \\
&&\le Cr_B\left(\fint_{8B}|\nabla f|^p\,d\mu\right)^{1/p},
\end{eqnarray*}
where the last inequality follows from the fact that $\mathcal{M}_q$ is $L^p$-bounded for $p>q$.
The above estimate completes the proof of $(\widetilde{P_p})$ and therefore the theorem.

Let us prove the claim.  Take $x_{M_0}\in M_0\cap 2B$ and set $B_{x_{M_0}}=B(x_{M_0},1)$.
 Note that $B(x_{M_0},1)\subset 3B$ since $r_B>100$. Recall that we assume $\mathrm{diam}(M_0)=1$.
For all $x,y\in B$, we write
\begin{equation}\label{2.1x}
|f(x)-f(y)|\le |f(x)-f_{B_{x_{M_0}}}|+|f(y)-f_{B_{x_{M_0}}}|.
\end{equation}

{\bf Step 1. } Suppose first that $d(x,x_{M_0})\le 100$. We choose a sequence of balls $\{B_{j}\}_{j=0}^\infty$ such that
$B_j=B(x,2^{-j}*102)$ for each $j\ge 0$. As $x\in B$ and $r_B>100$,
we have $B_j\subset 3B\subset 8B$. We write
\begin{eqnarray}\label{2.1}
\left|f(x)-f_{B_{x_{M_0}}}\right| &&\le \left|f(x)-f_{B_0}\right|+|f_{B_{0}}-f_{B_{x_{M_0}}}|.
\end{eqnarray}

For the first term, note that $B_{j+1}\subset B_j$ for each $j\ge 0$. By using $(P_{2,\loc})$, $q>N\vee 2$ and the H\"older inequality, we conclude that
\begin{eqnarray}\label{2.2}
\left|f(x)-f_{B_0}\right|&&=\lim_{j\to \infty} \left|f_{B_j}-f_{B_0}\right|\le \sum_{j=0}^{\infty}\left|f_{B_j}-f_{B_{j+1}}\right|
\le \sum_{j=0}^{\infty} \fint_{B_{j+1}}\left|f-f_{B_{j}}\right|\,d\mu\nonumber\\
&&\le C\sum_{j=0}^{\infty} \fint_{B_{j}}\left|f-f_{B_{j}}\right|\,d\mu\le \sum_{j=0}^{\infty} C2^{-j}*102\left(\fint_{B_{j}}\left|\nabla f\right|^q\,d\mu\right)^{1/q}\nonumber\\
&&\le \sum_{j=0}^{\infty} C2^{-j} \mathcal{M}_q\left(|\nabla f|\chi_{8B}\right)(x)\nonumber\\
&&\le Cr_B\mathcal{M}_q\left(|\nabla f|\chi_{8B}\right)(x),
\end{eqnarray}
where in the last step we used the fact $r_B>100$.

For the remaining term in \eqref{2.1}, note that $B_{x_{M_0}}=B(x_{M_0},1)\subset B(x,102)=B_0\subset 3B\subset 8B$ since $d(x,x_{M_0})\le 100$.
From this and using $(D_N)$, $(P_{2,\loc})$, $q>N\vee 2$ and the H\"older inequality, we conclude that
\begin{eqnarray}\label{2.3}
\left|f_{B_{0}}-f_{B_{x_{M_0}}}\right|&&\le \fint_{B_{x_{M_0}}}\left|f-f_{B_{ 0}}\right|\,d\mu \le C\fint_{B_{0}}\left|f-f_{B_{0}}\right|\,d\mu \nonumber\\
&&\le C\left(\fint_{B_{0}}\left|\nabla f\right|^q\,d\mu\right)^{1/q}\le C\mathcal{M}_q\left(|\nabla f|\chi_{8B}\right)(x)\nonumber\\
&&\le Cr_B \mathcal{M}_q\left(|\nabla f|\chi_{8B}\right)(x),
\end{eqnarray}
since $r_B>100$. The estimates \eqref{2.2} and \eqref{2.3} yield that for $x\in B$ with $d(x,x_{M_0})\le 100$,
\begin{equation}\label{2.4}
|f(x)-f_{B_{x_{M_0}}}|\le Cr_B\mathcal{M}_q\left(|\nabla f|\chi_{8B}\right)(x).
\end{equation}

{\bf Step 2. }  Suppose $d(x,x_{M_0})> 100$ and let $k_0\in\cn$ be such that
\begin{equation}\label{2.5}
9\left(\frac 98\right)^{k_0}<d(x,x_{M_0})+8\le 9\left(\frac 98\right)^{k_0+1}.
\end{equation}
Note that \eqref{2.5} together with $d(x,x_{M_0})> 100$ implies
\begin{equation}\label{2.6}
8\left(\frac 98\right)^{k_0}< d(x,x_{M_0})<9\left(\frac 98\right)^{k_0+1}.
\end{equation}
Take a geodesic $\gamma$ connecting $x$ to $x_{M_0}$.
 \begin{figure}[ht]
\centerline{ \epsfig{file=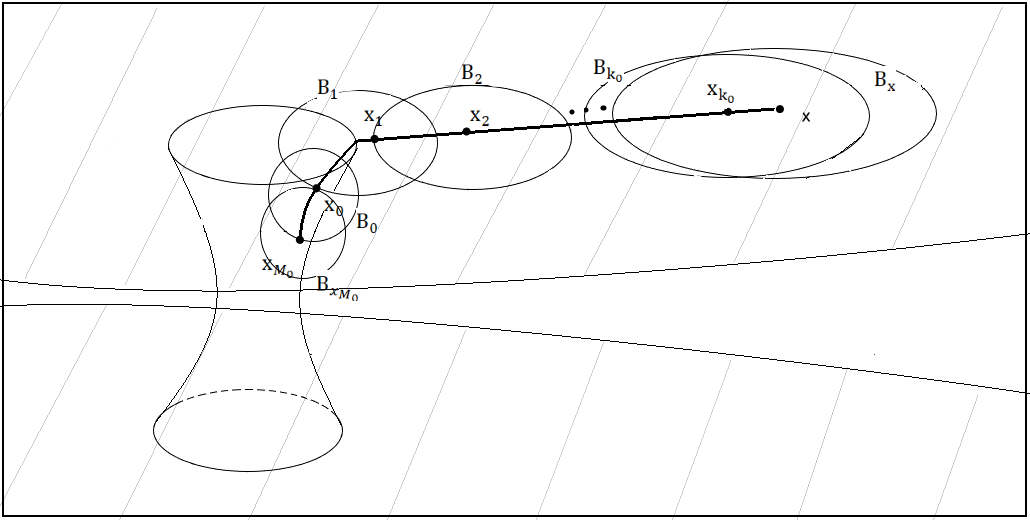, scale=0.4}
              }
             \caption{The chosen points and balls along the geodesic.}
\end{figure}
On the geodesic, we choose a sequence of points $\{x_{j}\}_{j=0}^{k_0}$ such that
$d(x_0,x_{M_0})=1$, $d(x_j,x_{j+1})=(9/8)^{j+1}$ for $0\le j\le k_0-1$.  As the set  $\{x_{j}\}_{j=1}^{k_0}$ belongs to the geodesic $\gamma$,
we have
$$d(x,x_{k_0})=d(x,x_{M_0})-\sum_{j=0}^{k_0}\left(\frac 98\right)^{j}=d(x,x_{M_0})-8\left[\left(\frac 98\right)^{k_0+1}-1\right],$$
which together with \eqref{2.5} yields that
\begin{equation}\label{2.7}
0<d(x,x_{k_0})\le \left(\frac 98\right)^{k_0+1}.
\end{equation}

 Let $B_j=B(x_j, (9/8)^{j} )$ for $0\le j\le k_0$ and $B_x=B(x,(9/8)^{k_0+1})$.  We write
\begin{eqnarray*}
\left|f(x)-f_{B_{x_{M_0}}}\right|&&\le \left|f_{B_{x_{M_0}}}-f_{B_0}\right|+\sum_{j=0}^{k_0-1}\left|f_{B_j}-f_{B_{j+1}}\right|+\left|f_{B_{x}}-f_{B_{k_0}}\right|+\left|f(x)-f_{B_x}\right|\\
&&=:I+II+III+IV.
 \end{eqnarray*}

For the term $I$, by $(P_{2,\loc})$, $(D_N)$  and the H\"older inequality, we obtain
\begin{eqnarray*}
I&&\le \left|f_{B_{x_{M_0}}}-f_{B(x_{M_0},2)}\right|+\left|f_{B_0}-f_{B(x_{M_0},2)}\right|\\
&&\le \fint_{B_{x_{M_0}}}|f-f_{B(x_{M_0},2)}|\,d\mu+\fint_{B_{0}}|f-f_{B(x_{M_0},2)}|\,d\mu\\
&&\le \frac{C}{V(x_{M_0},2)}\int_{B(x_{M_0},2)}|f-f_{B(x_{M_0},2)}|\,d\mu\\
&&\le \frac{C}{V(x_{M_0},2)^{1/2}}\left(\int_{B(x_{M_0},2)}|\nabla f|^2\,d\mu\right)^{1/2}\le \frac{C}{V(x_{M_0},2)^{1/q}}\left(\int_{B(x_{M_0},2)}|\nabla f|^q\,d\mu\right)^{1/q}\\
&&\le \frac{Cd(x,x_{M_0})^{N/q}}{V(x_{M_0},2d(x,x_{M_0}))^{1/q}} \left(\int_{B(x,d(x,x_{M_0})+2)}|\nabla f|^q\,d\mu\right)^{1/q}\\
&&\le \frac{Cd(x,x_{M_0})^{N/q}}{V(x,2d(x,x_{M_0}))^{1/q}} \left(\int_{B(x,2d(x,x_{M_0}))}|\nabla f|^q\,d\mu\right)^{1/q}\\
&&\le Cd(x,x_{M_0})^{N/q} \mathcal{M}_q\left(|\nabla f|\chi_{8B}\right)(x)
\le  Cr_B \mathcal{M}_q\left(|\nabla f|\chi_{8B}\right)(x).
\end{eqnarray*}
Above, in the last second inequality we used that $d(x,x_{M_0})<3r_B$, $B(x, 2d(x,x_{M_0}))\subset 8B$ since  $x\in B$ and $x_{M_0}\in 2B$, and in the last inequality we used that $100<d(x,x_{M_0})<3r_B.$

Let us estimate the second term $II$. For each $0\le j<k_0$,
we have $B_j,B_{j+1}\subset B(x_{j+1}, \frac{17}{8}(\frac 98)^{j})$ and $B(x_{j+1}, \frac{17}{8}(\frac 98)^{j})\subset B(x,2d(x,x_{M_0}))$ by \eqref{2.6}. By the doubling property, we have
\begin{eqnarray*}
\left|f_{B_j}-f_{B_{j+1}}\right| &&\le \left|f_{B_j}-f_{B(x_{j+1}, \frac{17}{8}(\frac 98)^{j}) }\right|+\left|f_{B_{j+1}}-f_{B(x_{j+1}, \frac{17}{8}(\frac 98)^{j})}\right|\nonumber\\
&&\le \fint_{B_j}|f-f_{B(x_{j+1},\frac{17}{8}(\frac 98)^{j})}|\,d\mu+\fint_{B_{j+1}}|f-f_{B(x_{j+1},\frac{17}{8}(\frac 98)^{j})}|\,d\mu\nonumber\\
&&\le \frac{C}{V(x_{j+1}, \frac{17}{8}(\frac 98)^{j})}\int_{B(x_{j+1},\frac{17}{8}(\frac 98)^{j})}|f-f_{B(x_{j+1},\frac{17}{8}(\frac 98)^{j})}|\,d\mu.
\end{eqnarray*}

For the $j$'s such that
$$\left(\frac 98\right)^{j}+\left(\frac 98\right)^{j+1}=\frac{17}{8}\left(\frac 98\right)^{j}\le 100,$$
by using $(P_{2,\loc})$, $(D_N)$, \eqref{2.6} and the H\"older inequality, we conclude that
\begin{eqnarray}\label{2.8}
\left|f_{B_j}-f_{B_{j+1}}\right|
&&\le \frac{C\frac{17}{8}(\frac 98)^{j}}{V(x_{j+1}, \frac{17}{8}(\frac 98)^{j})^{1/q}}\left(\int_{B(x_{j+1}, \frac{17}{8}(\frac 98)^{j})}|\nabla f|^q\,d\mu\right)^{1/q}\nonumber\\
&&\le \frac{C(\frac 98)^{j+(k_0-j)N/q} }{V(x_{j+1}, (\frac 98)^{k_0})^{1/q}} \left(\int_{B(x,2d(x,x_{M_0}))}|\nabla f|^q\,d\mu\right)^{1/q}\nonumber\\
&&\le \frac{C(\frac 98)^{j(1-N/q)}d(x,x_{M_0})^{N/q}} {V(x,  2d(x,x_{M_0}))^{1/q}} \left(\int_{B(x,2d(x,x_{M_0}))}|\nabla f|^q\,d\mu\right)^{1/q},
\end{eqnarray}
where in the last inequality we used that $V(x_{j+1}, (\frac 98)^{k_0})\sim V(x_{j+1}, 2d(x,x_{M_0}))\sim V(x, 2d(x,x_{M_0}))$
which follows from \eqref{2.6}.

For the $j$'s such that
$$\left(\frac 98\right)^{j}+\left(\frac 98\right)^{j+1}=\frac{17}{8}\left(\frac 98\right)^{j}>100,$$
notice that
$$d(x_{j+1},x_{M_0})=\sum_{j=0}^{j+1}\left(\frac 98\right)^{k} =\frac{\left(\frac 98\right)^{j+2}-1}{\frac{9}{8}-1}=8\left(\frac 98\right)^{j+2}-8,$$
which together with $\mathrm{diam}(M_0)=1$ implies
$$\dist(x_{j+1},M_0)-2\left[\left(\frac 98\right)^{j}+\left(\frac 98\right)^{j+1}\right]\ge d(x_{j+1},x_{M_0})-1-\frac{17}{4}\left(\frac 98\right)^{j}\ge \frac{47}{8}\left(\frac 98\right)^{j}-9>\frac{47}{8} \frac{800}{17}-9>100. $$
Therefore, $2B(x_{j+1},\frac{17}{8}\left(\frac 98\right)^{j} )\cap M_0=\emptyset$.
Applying $(P^E_2)$, $(D_N)$, \eqref{2.6} and the H\"older inequality, we conclude that for such $j$'s
\begin{eqnarray}\label{2.9}
\left|f_{B_j}-f_{B_{j+1}}\right| &&\le \frac{C\frac{17}{8}(\frac 98)^{j}}{V(x_{j+1}, \frac{17}{8}(\frac 98)^{j})^{1/q}}\left(\int_{B(x_{j+1}, \frac{17}{8}(\frac 98)^{j})}|\nabla f|^q\,d\mu\right)^{1/q}\nonumber\\
&&\le \frac{C(\frac 98)^{j+(k_0-j)N/q} }{V(x_{j+1}, (\frac 98)^{k_0})^{1/q}} \left(\int_{B(x,2d(x,x_{M_0}))}|\nabla f|^q\,d\mu\right)^{1/q}\nonumber\\
&&\le \frac{C(\frac 98)^{j(1-N/q)}d(x,x_{M_0})^{N/q}} {V(x,  2d(x,x_{M_0}))^{1/q}} \left(\int_{B(x,2d(x,x_{M_0}))}|\nabla f|^q\,d\mu\right)^{1/q}.
\end{eqnarray}
Combining \eqref{2.8} and \eqref{2.9}, we further deduce from \eqref{2.6} and the fact $q>N$ that
\begin{eqnarray*}
II&&\le \sum_{j=0}^{k_0-1} \frac{C(\frac 98)^{j(1-N/q)}d(x,x_{M_0})^{N/q}} {V(x,  2d(x,x_{M_0}))^{1/q}} \left(\int_{B(x,2d(x,x_{M_0}))}|\nabla f|^q\,d\mu\right)^{1/q}\\
&&\le \sum_{j=0}^{k_0-1} C(\frac 98)^{j(1-N/q)}d(x,x_{M_0})^{N/q}\mathcal{M}_q\left(|\nabla f|\chi_{8B}\right)(x) \\
&&\le C(\frac 98)^{k_0(1-N/q)}d(x,x_{M_0})^{N/q} \mathcal{M}_q\left(|\nabla f|\chi_{8B}\right)(x)\\
&&\le Cr_B \mathcal{M}_q\left(|\nabla f|\chi_{8B}\right)(x).
\end{eqnarray*}

For the term $III$,  by the choice of the points $\{x_{j}\}_{j=0}^{k_0}$, we see that
\begin{eqnarray*}
d(x,x_{M_0})-2\left[\left(\frac 98\right)^{k_0+1}+\left(\frac 98\right)^{k_0}\right]&&\ge \sum_{j=0}^{k_0}{\left(\frac 98\right)^{j}} -
\frac{17}{4}\left(\frac 98\right)^{k_0}
\ge \frac{\left(\frac 98\right)^{k_0+1}-1}{\frac 98-1}-\frac{17}{4}\left(\frac 98\right)^{k_0}\ge \frac{19}{4} \left(\frac 98\right)^{k_0}-8.
\end{eqnarray*}
By \eqref{2.5} one has $108<d(x,x_{M_0})+8\le 9\left(9/8\right)^{k_0+1}$, which implies
$$\frac{19}{4} \left(\frac 98\right)^{k_0}-8=\frac{19}{4}\frac{8}{9} \left(\frac 98\right)^{k_0+1}-8>12\frac{38}{9}-8>40,$$
and hence,
\begin{equation}\label{2.10}
\dist (x,M_0)-2\left[\left(\frac 98\right)^{k_0+1}+\left(\frac 98\right)^{k_0}\right]\ge d(x,x_{M_0})-1-2\left[\left(\frac 98\right)^{k_0+1}+\left(\frac 98\right)^{k_0}\right]\ge 39,
\end{equation}
from which it follows that $2B(x,\frac{17}{8}(\frac 98)^{k_0})\cap M_0=\emptyset$. By \eqref{2.6},
$B(x,\frac{17}{8}(\frac 98)^{k_0})\subset B(x,d(x,x_{M_0}))$, where $d(x,x_{M_0})<3r_B$.
Thus,
by applying $(P^E_2)$, $(D_N)$ and the H\"older inequality, we conclude that
\begin{eqnarray*}
III&&\le \left| f_{B_{x}}-f_{B(x,\frac{17}{8}(\frac 98)^{k_0})}\right|+\left|f_{B_{k_0}}-f_{B(x,\frac{17}{8}(\frac 98)^{k_0})}\right|\\
&&\le \frac{C}{V(x, \frac{17}{8}(\frac 98)^{k_0})}\int_{B(x,\frac{17}{8}(\frac 98)^{k_0})}|f-f_{B(x,\frac{17}{8}(\frac 98)^{k_0})}|\,d\mu\nonumber\\
&&\le \frac{C (\frac 98)^{k_0}}{V(x,\frac{17}{8}(\frac 98)^{k_0})^{1/q}}\left(\int_{B(x,\frac{17}{8}(\frac 98)^{k_0})}|\nabla f|^q\,d\mu\right)^{1/q}\nonumber\\
&&\le Cr_B\mathcal{M}_q\left(|\nabla f|\chi_{8B}\right)(x).
\end{eqnarray*}

For the term $IV$, \eqref{2.10} implies $2B_x\cap M_0=\emptyset$, where $B_x=B(x,(9/8)^{k_0+1})$. Therefore, by applying $(P^E_2)$, $(D_N)$ and the approach similar to \eqref{2.2}, we find
\begin{eqnarray*}
IV=|f(x)-f_{B_x}|&&\le Cr_B\mathcal{M}_q\left(|\nabla f|\chi_{8B}\right)(x).
\end{eqnarray*}

For $x\in B$ with $d(x,x_{M_0})>100$, from the estimates for $I,II,III,IV$, it follows that
\begin{eqnarray}\label{2.11}
|f(x)-f_{B_{x_{M_0}}}|&&\le Cr_B\mathcal{M}_q\left(|\nabla f|\chi_{8B}\right)(x).
\end{eqnarray}

Apparently, the same proofs in Step 1 and Step 2 work for $y\in B$ and yield that
\begin{eqnarray}\label{2.12}
|f(y)-f_{B_{x_{M_0}}}|&&\le Cr_B\mathcal{M}_q\left(|\nabla f|\chi_{8B}\right)(y).
\end{eqnarray}

A combination of \eqref{2.1x}, \eqref{2.4}, \eqref{2.11} and \eqref{2.12} completes the proof of the Claim.
\end{proof}

\begin{rem}\label{rem-telescopic}\rm
The approach used in proving \eqref{2.4} is called ``telescopic approach" in the literature, see \cite[p. 211, proof of Theorem 8.1.7, and p. 243, Section 8.5]{hkst}.
\end{rem}

\begin{prop}\label{exp-poincare}
Assume that $(D_{N})$ holds on $M$  with $0<N<\infty$ and that $(P_{2,\loc})$ holds.
Then there exist $N_\mu>0$ and $C>0$ such that for any ball $B=B(x_0,r)$, $r>1$, and any $f\in C^1(\overline {2B})$, it holds
$$\fint_B|f-f_B|^2\,d\mu\le Cr^{2N_\mu+2+N}\fint_{2B}|\nabla f|^2\,d\mu.\leqno(P_{G})$$
\end{prop}
\begin{proof}
Based on the validity of $(P_{2,\loc})$, we only need to show $(P_{G})$ for balls $B=B(x_0,r)$
when $r$ is sufficiently large. Let us assume $r>100$. Set $B_0=B(x_0,1)$.
By $(D_N)$, we can find a sequence of  balls via an $\epsilon$-net with $\epsilon=1/2$, $\{B_i\}_{1\le i\le C(r)}$, where $C(r)$ is an integer not bigger than $Cr^{N_\mu}$,
$N_\mu>0$,
such that each ball $B_i$ is of radius one and the center of $B_i$ is located in $B$, $\frac 12 B_i\cap \frac 12 B_j=\emptyset$ for any $i\neq j$, $0\le i,j\le C(r)$; see \cite[p. 102]{hkst} for instance.
By the choice of $B_i$ we have
\begin{equation}\label{2.13}
B\subset \cup_{i=0}^{C(r)}B_i\subset \cup_{i=0}^{C(r)}3B_i\subset B(x_0,r+3)\subset 2B.
\end{equation}

Write
\begin{eqnarray*}
\left(\fint_B|f-f_B|^2\,d\mu\right)^{1/2}\le 2\left(\fint_B|f-f_{B_0}|^2\,d\mu\right)^{1/2}\le \frac{2}{\mu(B)^{1/2}}\sum_{i=0}^{C(r)}\left(\int_{B_i}|f-f_{B_0}|^2\,d\mu\right)^{1/2}.
\end{eqnarray*}
If  $i=0$, then $(P_{2,\loc})$ implies
$$\int_{B_0}|f-f_{B_0}|^2\,d\mu\le C\int_{B_0}|\nabla f|^2\,d\mu.$$
For other $i$'s, let $x_i$ be the center of the ball $B_i$, and there exists a geodesic $\gamma(x_0,x_i)$ that links $x_i$ to $x_0$ with length equaling $d(x_0,x_i)$. As $x_i\in B$, one has $d(x_0,x_i)<r$ and $\gamma(x_0,x_i)\subset B$.
Along $\gamma(x_0,x_i)$, we may find a sequence of balls $\{B_{i,j}\}_{1\le j\le C(i)}$ with $C(i)$ be an integer not bigger than $2d(x_0,x_i)$, such that
each ball $B_{i,j}$ is of radius one and has center on $\gamma(x_0,x_i)$,
and $B_{i,1}\cap B_0\neq\emptyset$, $B_{i,C(i)}\cap B_i\neq \emptyset$ and
 $B_{i,j}\cap B_{i,j-1}\neq\emptyset$ if $2\le j\le C(i)$.
As the balls  $B_{i,j}$ are of radius one and have center on $\gamma(x_0,x_i)$, where $\gamma(x_0,x_i)\subset B$,
 we have
\begin{equation}\label{2.14}
\cup_{1\le j\le C(i)}B_{i,j}\subset\cup_{1\le j\le C(i)}3B_{i,j}\subset B(x_0,r+3)\subset 2B.
\end{equation}

 A chain argument implies
 \begin{eqnarray*}
   \left(\fint_{B_i}|f-f_{B_0}|^2\,d\mu\right)^{1/2}&&\le  \left(\fint_{B_i}|f-f_{B_i}|^2\,d\mu\right)^{1/2}+|f_{B_i}-f_{B_{i,C(i)}}|+|f_{B_0}-f_{B_{i,1}}|+\sum_{j=2}^{C(i)}|f_{B_{i,j}-f_{B_{i,j-1}}}|.
 \end{eqnarray*}
 Notice that, as ${B_0}\cap {B_{i,1}}\neq\emptyset$, $B_{i,1}\subset 3B_0$, and therefore,
\begin{eqnarray*}
 |f_{B_0}-f_{B_{i,1}}|&&\le |f_{3B_0}-f_{B_0}|+|f_{3B_0}-f_{B_{i,1}}|\le C\fint_{3B_0}|f-f_{3B_0}|\,d\mu \le C\left(\fint_{3B_0}|\nabla f|^2\,d\mu\right)^{1/2}.
  \end{eqnarray*}
Similarly, we conclude via $(D_N)$, \eqref{2.13} and \eqref{2.14} that  for each $1\le i\le C(r)$,
  \begin{eqnarray*}
   &&\left(\fint_{B_i}|f-f_{B_0}|^2\,d\mu\right)^{1/2}\\
   &&\le C\left(\fint_{3B_0}|\nabla f|^2\,d\mu\right)^{1/2}+ C\left(\fint_{3B_{i,C(i)}}|\nabla f|^2\,d\mu\right)^{1/2} +C\sum_{j=1}^{C(i)-1}\left(\fint_{3B_{i,j}}|\nabla f|^2\,d\mu\right)^{1/2} \\
    &&\le C\frac{r^{N/2}}{\mu(2B)^{1/2}}\left[\left(\int_{3B_0}|\nabla f|^2\,d\mu\right)^{1/2}+ \left(\int_{3B_{i,C(i)}}|\nabla f|^2\,d\mu\right)^{1/2} +\sum_{j=1}^{C(i)-1}\left(\int_{3B_{i,j}}|\nabla f|^2\,d\mu\right)^{1/2}\right] \\
   &&\le C\frac{C(i) r^{N/2}}{\mu(2B)^{1/2}}\left(\int_{2B}|\nabla f|^2\,d\mu\right)^{1/2}\le C\frac{r^{N/2+1}}{\mu(2B)^{1/2}}\left(\int_{2B}|\nabla f|^2\,d\mu\right)^{1/2}.
 \end{eqnarray*}
 Summarizing these estimates, we conclude that
 \begin{eqnarray*}
\left(\fint_B|f-f_B|^2\,d\mu\right)^{1/2}&&\le \frac{2}{\mu(B)^{1/2}}\sum_{i=0}^{C(r)}\left(\int_{B_i}|f-f_{B_0}|^2\,d\mu\right)^{1/2}\\
&&\le \frac{C}{\mu(B)^{1/2}}\left[\left(\int_{3B_0}|\nabla f|^2\,d\mu\right)^{1/2}+\sum_{i=1}^{C(r)}\frac{r^{N/2+1}\mu(B_i)^{1/2}}{\mu(2B)^{1/2}}\left(\int_{2B}|\nabla f|^2\,d\mu\right)^{1/2}\right]\\
&&\le \frac{CC(r)r^{N/2+1}}{\mu(B)^{1/2}}\left(\int_{2B}|\nabla f|^2\,d\mu\right)^{1/2}\\
&&\le {Cr^{N_\mu+N/2+1}}\left(\fint_{2B}|\nabla f|^2\,d\mu\right)^{1/2},
\end{eqnarray*}
where in the third inequality we used $\mu(B_i)/\mu(2B)\le 1$. This gives the desired estimate.
\end{proof}

\begin{thm}\label{asy-poin}
If there exists $C_M>0$ such that for each $x\in M$,  it holds
$$Ric_M(x)\ge -\frac{C_M}{[d(x,x_M)+1]^2},$$
then $(P^E_2)$ holds on $M$.
\end{thm}
\begin{proof}
By Buser's inequality (cf. \cite{bu82,hak}), there exists a constant $C>0$ depending only on the dimension such that,
for any $f\in C^1(\bar{B})$,
$$\int_B|f-f_B|\,d\mu\le Ce^{\sqrt{K}r_B}r_B\int_B|\nabla f|\,d\mu;$$
where $K\ge 0$ and the Ricci curvature on $B$ is not less than $-K$.

For any $B\subset M$ with $2B\cap M_0=\emptyset$, we then have
$$Ric_M(x)\ge -\frac{C_M}{[r_B+1]^2},\ \forall \ x\in B.$$
This together with Buser's inequality implies
$$\int_B|f-f_B|\,d\mu\le Ce^{\sqrt{\frac{C_M}{r_B^2}}r_B}r_B\int_B|\nabla f|\,d\mu\le Cr_B\int_B|\nabla f|\,d\mu,$$
which together with \cite[Theorem 5.1]{hak} further implies
$$\fint_B|f-f_B|^2\,d\mu\le Cr_B^2\fint_B|\nabla f|^2\,d\mu,$$
as desired.
\end{proof}

\section{Riesz transform for $p$ below the lower dimension}

\subsection{Riesz transform via heat kernel regularity}
\hskip\parindent In this section, we study the behavior of the Riesz transform on $L^p(M)$, where $p\in (2,n)$.
In what follows, let $A_r:=I-(I-e^{-r^2\mathcal L})^m$, where $m\in\cn$ is chosen such that $m>N/4$; see \cite[p. 932]{acdh}.
Let $T:=\nabla \mathcal{L}^{-1/2}$. The sharp maximal function $\mathcal{M}^\#_{T,A}f$ for every locally integrable function $f$ is given as
$$\mathcal{M}^\#_{T,A}f(x):=\sup_{B:\, x\in B}\left(\fint_{B}|T(1-A_{r_B})f|^2\,d\mu\right)^{1/2}.$$

\begin{lem}\label{sharp-maximal}
Assume that $(D_{N})$ holds on $M$  with $0< N<\infty$. There exists $C>0$ such that for any $f\in L^2(M)$,   any ball $B$ and $x\in B$, it holds
\begin{equation}
\left(\fint_{B}|T(I-A_{r_B})f|^2\,d\mu\right)^{1/2}\le \mathcal{M}^\#_{T,A}f(x)\le C\mathcal{M}_2(|f|)(x).
\end{equation}
\end{lem}
\begin{proof}
  See \cite[Lemma 3.1]{acdh}.
\end{proof}

\begin{lem}\label{key-lemma}
Assume that $(D_{N})$ holds on $M$  with $0< N<\infty$, and that
$(UE)$ and $(G_{p_0})$ for some $p_0\in (2,\infty)$  hold. Then for every $p\in (2,p_0)$,
there exist $C,\tau>0$ such that  for every ball $B$ with radius $r_B$ and every
 $f\in L^2(M)$ supported in $U_i =2^{i+1}B\setminus 2^iB$,
$i\ge 2$, or $U_1 =4B$, one has
\begin{equation}\label{key-good}
\left(\fint_{B}|\nabla A_{r_B}f|^p\,d\mu\right)^{1/p}\le \frac{Ce^{-\tau 4^i}}{r_B}\left(\frac{1}{\mu(2^iB)}\int_{U_i}|f|^2\,d\mu\right)^{1/2}.
\end{equation}
\end{lem}
\begin{proof}
The lemma was proved in \cite[Lemma 3.2]{acdh}. Notice that, although in the statement of \cite[Lemma 3.2]{acdh},
 $(P_2)$ was assumed, its proof indeed only needs $(D_N)$, $(UE)$ and $(G_{p_0})$.
\end{proof}

Recall that $x_M\in M_0$ is fixed, and we assume that $\mathrm{diam}(M_0)=1$.
The reverse doubling condition only requires that for all $1<r<R<\infty$ it holds
$$\left(\frac{R}{r}\right)^n\lesssim \frac{V(x_M,R)}{V(x_M,r)}. \leqno(RD_{n})$$

\begin{lem}\label{potential}
Assume that $(D_{N})$ and $(RD_n)$ hold on $M$  with $1<n\le N<\infty$, and that $(UE)$ holds.
Let $C_0>10$ be fixed.
Then for any $p\in (1,n)$, there exists $C>0$, depending only on $C_0,n,N$,
such that  for any ball $B$, with $r_B>1$ and $C_0B\cap M_0\neq \emptyset$, and any $f\in L^p(M)$, it holds
$$\left(\fint_{B}|\mathcal{L}^{-1/2}f|^p\,d\mu\right)^{1/p}\le \frac{Cr_B}{\mu(B)^{1/p}}\|f\|_p.$$
\end{lem}
\begin{proof}
For each $x\in B$, write
\begin{eqnarray*}
|{\mathcal{L}}^{-1/2}f(x)|&&\le \frac{\sqrt{\pi}}{2}\int_0^{(2C_0r_B)^2}\left|e^{-s{\mathcal{L}}}f(x)\right|\frac{\,ds}{\sqrt s}+\frac{\sqrt{\pi}}{2}\int_{(2C_0r_B)^2}^\infty\left|e^{-s{\mathcal{L}}}f(x)\right|\frac{\,ds}{\sqrt s}=:I_1+I_2.
\end{eqnarray*}
For the term $I_1$, one has via the Minkowski inequality that
$$\left(\int_{B}|I_1|^p\,d\mu\right)^{1/p}\le C\int_0^{(2C_0r_B)^2}\|e^{-s{\mathcal{L}}}f\|_p\frac{\,ds}{\sqrt s}\le  C\int_0^{(2C_0r_B)^2}\|f\|_p\frac{\,ds}{\sqrt s}\le Cr_B\|f\|_p.$$
For the term $I_2$, notice that by $(D_{N})$, the assumptions $C_0B\cap M_0\neq\emptyset$ and $\diam(M_0)=1$, it holds
$$V(x_B,\sqrt{t})\sim V(x_M,\sqrt t)\sim V(x,\sqrt t)$$
for any $x\in B$ and $t\ge (2C_0r_B)^2$.
From this together with $(UE)$, $(RD_n)$ and the H\"older inequality, we deduce that for each $x\in B$
\begin{eqnarray*}
|I_2|&&\le  \int_{(2C_0r_B)^2}^\infty\int_M \frac{C}{V(x,\sqrt {s})}\exp\lf\{-\frac{d^2(x,y)}{cs}\r\}|f(y)|\,d\mu(y)\frac{\,ds}{\sqrt s}\\
&&\le C\int_{(2C_0r_B)^2}^\infty \|f\|_p\left(\int_M \frac{1}{V(x,\sqrt {s})^{p/(p-1)}}\exp\lf\{-\frac{pd^2(x,y)}{c(p-1)s}\r\}\,d\mu(y)\right)^{(p-1)/p}\frac{\,ds}{\sqrt s}\\
&&\le C\|f\|_p\int_{(2C_0r_B)^2}^\infty  \frac{C}{V(x_M,\sqrt {s})^{1/p}}\frac{\,ds}{\sqrt s}\\
&&\le C\|f\|_p\int_{(2C_0r_B)^2}^\infty  \frac{Cr_B^{n/p}}{V(x_M,2C_0r_B)^{1/p}s^{n/(2p)}}\frac{\,ds}{\sqrt s} \le \frac{Cr_B}{\mu(B)^{1/p}}\|f\|_p,
\end{eqnarray*}
where in the last inequality we used the fact that $p<n$.
Combining the estimates of $I_1$ and $I_2$, we conclude that
\begin{eqnarray*}
\left(\fint_{B}|\mathcal{L}^{-1/2}f|^p\,d\mu\right)^{1/p}&&\le  \left(\fint_{B}|I_1+I_2|^p\,d\mu\right)^{1/p}\le \frac{Cr_B}{\mu(B)^{1/p}}\|f\|_p,
\end{eqnarray*}
as desired.
\end{proof}

Using the the previous mapping property of the Riesz potentials  together with  Lemma \ref{key-lemma}, we deduce the following estimates.

\begin{prop}\label{key-estimate}
Assume that $(D_{N})$ and $(RD_n)$ hold on $M$  with $2<n\le N<\infty$,  and that $(UE)$ and $(P^E_2)$ hold.
 Suppose that $(G_{p_0})$ for some $p_0\in (2,n)$ holds. Let $p\in (2,p_0)$ and $q\in (2,n)$.
Let $10<C_0,\alpha,\beta<\infty$. Then for each $B=B(x_B,r_B)\subset M$ and each $f\in C^\infty_c(M)$, the followings hold.

(i) If ${r_B}<\alpha$, then there exists $C_1=C_1(n,N,p,p_0,\alpha)$ such that it holds
\begin{equation}
\left(\fint_{B}|TA_{r_B}f|^p\,d\mu\right)^{1/p}\le C_{1}\inf_{y\in B}\mathcal{M}_2(|Tf|)(y).
\end{equation}

(ii) If ${r_B}\ge \beta$ and $d(x_B,x_M)<C_0{r_B}$, then  there exists $C_2=C_2(n,N,p,q,p_0,C_0,\beta)$ such that it holds
\begin{eqnarray}\label{good-lambda-large-1}
\left(\fint_{B}|TA_{r_B}f|^p\,d\mu\right)^{1/p}\le  \frac{C_2}{\mu(B)^{1/q}}\|f\|_q.
\end{eqnarray}

(iii) If ${r_B}\ge \beta$ and $d(x_B,x_{M})\ge C_0{r_B}$, then there exist $C_3,\,C_4$, depending on $n,N,p,q,p_0,C_0,\beta$, such that  it holds
\begin{eqnarray}\label{good-lambda-large-2}
\left(\fint_{B}|TA_{r_B}f|^p\,d\mu\right)^{1/p}&&\le \frac{C_3\|f\|_q}{V(x_M, d(x_B,x_M)+1)^{1/q}}+C_{4}\inf_{y\in B}\mathcal{M}_2(|Tf|)(y).
\end{eqnarray}
\end{prop}
\begin{proof}  Let $g:=\mathcal{L}^{-1/2}f$ for $f\in C^\infty_c(M)$. Let $U_i =2^{i+1}B\setminus 2^iB$, $i\ge 2$, and $U_1 =4B$.

(i) The case ${r_B}<\alpha$ follows from a proof similar to  \cite[p. 935]{acdh}, we provide a proof for completeness.
By Lemma  \ref{key-good} one has
\begin{eqnarray*}
\left(\fint_{B}|TA_{r_B}f|^p\,d\mu\right)^{1/p}&&\le \sum_{i\ge 1}\left(\fint_{B}|\nabla A_{r_B} [(g-g_{4B})\chi_{U_i}]|^p\,d\mu\right)^{1/p}\\
&&\le \sum_{i\ge 1} \frac{Ce^{-\tau 4^i}}{{r_B}}\left(\frac{1}{\mu(2^iB)}\int_{U_i}|g-g_{4B}|^2\,d\mu\right)^{1/2}.
\end{eqnarray*}

By using $(P_{2,\loc})$ from Lemma \ref{pend-ploc}
and  $(P_G)$ from Proposition \ref{exp-poincare}, one finds
\begin{eqnarray*}
\left(\fint_{B}|TA_{r_B}f|^p\,d\mu\right)^{1/p}&&\le \sum_{i\ge 1} \frac{Ce^{-\tau 4^i}}{{r_B}}\left(\frac{1}{\mu(2^iB)}\int_{U_i}|g-g_{4B}|^2\,d\mu\right)^{1/2}\\
&&\le \sum_{i\ge 1}\frac{Ce^{-\tau 4^i}}{{r_B}}\left[\left(\fint_{2^{i+1}B}|g-g_{2^{i+1}B}|^2\,d\mu\right)^{1/2}+
\sum_{j=2}^i|g_{2^jB}-g_{2^{j+1}B}|\right]\\
&&\le \sum_{i\ge 1}\frac{Ce^{-\tau 4^i}}{{r_B}}\sum_{j=2}^{i+1} (2^{j+1}{r_B}) [(2^{j+1}{r_B})\vee 1]^{N_\mu+\frac{N}{2}}\left(\fint_{2^{j+1}B}|\nabla g|^2\,d\mu\right)^{1/2}\\
&&\overset{{r_B}<\alpha}{\leq} \sum_{i\ge 1}Ce^{-\tau 4^i}\sum_{j=1}^{i+1} 2^{j(N_\mu+1+\frac{N}{2})}\inf_{y\in B}\mathcal{M}_2(|\nabla g|)(y)\\
&&\le C_1\inf_{y\in B}\mathcal{M}_2(|Tf|)(y).
\end{eqnarray*}

(ii) Suppose now $d(x_B,x_M)<C_0r_B$ and ${r_B}\ge \beta$.
By Lemma \ref{key-lemma}, Lemma \ref{potential} together with  the H\"older inequality, one has
\begin{eqnarray*}
\left(\fint_{B}|TA_{r_B}f|^p\,d\mu\right)^{1/p}&&\le \sum_{i\ge 1}\left(\fint_{B}|\nabla A_{r_B} (g\chi_{U_i})|^p\,d\mu\right)^{1/p}\le \sum_{i\ge 1} \frac{Ce^{-\tau 4^i}}{{r_B}}\left(\frac{1}{\mu(2^iB)}\int_{U_i}|g|^{2}\,d\mu\right)^{1/2}\\
&&\le \sum_{i\ge 1} \frac{Ce^{-\tau 4^i}}{{r_B}}\left(\frac{1}{\mu(2^iB)}\int_{U_i}|g|^{q}\,d\mu\right)^{1/q}\\
&&\le \sum_{i\ge 1} \frac{Ce^{-\tau 4^i}2^ir_B}{r_B\mu(2^iB)^{1/q}} \|f\|_q\le \frac{C}{\mu(B)^{1/q}} \|f\|_q.
\end{eqnarray*}


(iii)  If  $d(x_B,x_{M})\ge C_0{r_B}$ and ${r_B}\ge \beta$, then the ball $B$ is included in one end.
Let $k\in\cn$ such that $2^{k+1}{r_B}\le d(x_B,x_M)<2^{k+2}{r_B}$ (recall that $C_0>10$).

By Lemma  \ref{key-good} again one has
\begin{eqnarray*}
\left(\fint_{B}|TA_{r_B}f|^p\,d\mu\right)^{1/p}&&\le \sum_{i\ge 1}\left(\fint_{B}|\nabla A_{r_B} [(g-g_{4B})\chi_{U_i}]|^p\,d\mu\right)^{1/p}\\
&&\le \sum_{i= 1}^k \frac{Ce^{-\tau 4^i}}{{r_B}}\left(\frac{1}{\mu(2^iB)}\int_{U_i}|g-g_{4B}|^2\,d\mu\right)^{1/2}+\sum_{i>k}\cdots\\
&&=: I_1+I_2.
\end{eqnarray*}

Since $d(x_B,x_M)<2^{k+2}{r_B}$,  for each $i>k$, $2^{i+1}B\cap M_0\neq \emptyset$. By Lemma \ref{potential},
$q<n$, $(D_N)$ and $(RD_n)$, we obtain
\begin{eqnarray*}
I_2&&\le \sum_{i>k}\frac{Ce^{-\tau 4^i}}{{r_B}}\left(\frac{1}{\mu(2^iB)}\int_{U_i}|g-g_{4B}|^2\,d\mu\right)^{1/2}\\
&&\le \sum_{i>k}\frac{Ce^{-\tau 4^i}}{{r_B}}\left[|g|_{4B}+\left(\frac{1}{\mu(2^iB)}\int_{U_i}|g|^2\,d\mu\right)^{1/2}\right]\\
&&\le  \sum_{i>k}\frac{Ce^{-\tau 4^i}}{{r_B}}\left[\left(\frac{1}{\mu(B)}\int_{2^{k+1}B}|g|^{q}\,d\mu\right)^{1/q}+
\left(\frac{1}{\mu(2^iB)}\int_{U_i}|g|^{q}\,d\mu\right)^{1/q}\right]\\
&&\le  \sum_{i>k}Ce^{-\tau 4^i} \left[\frac {2^k}{\mu(B)^{1/q}}\|f\|_q+
\frac{2^i}{\mu(2^iB)^{1/q}}\|f\|_q\right]\\
&&\le  \sum_{i>k}C\|f\|_qe^{-\tau 4^i} \left[\frac {2^{k+kN/q}}{\mu(2^kB)^{1/q}}+
\frac{2^{i}}{V(x_M,2^{i}r_B)^{1/q}}\right]\\
&&\le  \sum_{i>k}C\|f\|_qe^{-\tau 4^i} \left[\frac {2^{k+kN/q}}{V(x_M,2^{k}r_B)^{1/q}}+
\frac{2^{k+(i-k)(1-n/q)}}{V(x_M,2^{k}r_B)^{1/q}}\right]\\
&&\le \frac{Ce^{-c\tau 4^k}2^{k+kN/q}}{V(x_M,d(x_B,x_M)+1)^{1/q}}\|f\|_q\\
&&\le \frac{C_3}{V(x_M,d(x_B,x_M)+1)^{1/q}}\|f\|_q.
\end{eqnarray*}
Using $(P^E_2)$, we can estimate the term $I_1$  as
\begin{eqnarray*}
I_1&&=\sum_{i=1}^k \frac{Ce^{-\tau 4^i}}{{r_B}}\left(\frac{1}{\mu(2^iB)}\int_{U_i}|g-g_{4B}|^2\,d\mu\right)^{1/2}\\
&&\le \sum_{i=1}^k \frac{Ce^{-\tau 4^i}}{{r_B}}\left[\left(\fint_{2^{i+1}B}|g-g_{2^{i+1}B}|^2\,d\mu\right)^{1/2}+
\sum_{j=2}^i|g_{2^jB}-g_{2^{j+1}B}|\right]\\
&&\le \sum_{i=1}^k \frac{Ce^{-\tau 4^i}}{{r_B}}\sum_{j=2}^{i+1} (2^j{r_B}) \left(\fint_{2^jB}|\nabla g|^2\,d\mu\right)^{1/2}\\
&&\le C_4\inf_{y\in B}\mathcal{M}_2(|Tf|)(y).
\end{eqnarray*}
Using the estimates of $I_1$ and $I_2$, one can finally conclude that
\begin{eqnarray*}
\left(\fint_{B}|TA_{r_B}f|^p\,d\mu\right)^{1/p}&&\le  \frac{C_3}{V(x_M,d(x_B,x_M)+1)^{1/q}}\|f\|_q+C_{4}\inf_{x\in B}\mathcal{M}_2(|Tf|)(x),
\end{eqnarray*}
as desired.
\end{proof}

Using Proposition \ref{key-estimate} and adapting the argument from \cite{acdh},
we are able to provide a modified good-$\lambda$ inequality. The key ingredient is that for large balls,
Proposition \ref{key-estimate} allows us to deduce a small error term in the good-$\lambda$ inequality; see Proposition \ref{large-ball} below.

\begin{prop}\label{small-ball}
Assume that $(D_{N})$ and $(RD_n)$ hold on $M$  with $2<n\le N<\infty$,  and that $(UE)$ and $(P^E_2)$ hold.
 Assume that $(G_{p_0})$ holds for some $p_0\in (2,n)$ . Let $\alpha>10$  and $2<q<p_0$.
There exist $K_0,C>0$ only depending on $n,N,\alpha,q,p_0$, such that for
each $f\in C^\infty_c(M)$,  every $\lambda>0$, $K>K_0$ and $\gamma>0$,
and every ball $B_0=B(x_B,r_B)$, ${r_B}<\alpha$, if there exists $x_0\in B_0$ such that
 $\mathcal{M}_2(|Tf|)(x_0)\le \lambda$, then it holds
\begin{equation}
\mu\left(\left\{x\in B_0:\,\mathcal{M}_2(|Tf|)(x)>K\lambda,\, \mathcal{M}^\#_{T,A}f(x)\le \gamma \lambda\right\}\right)\le C({\gamma^2}+K^{-q})\mu(B_0).
\end{equation}
\end{prop}
\begin{proof}
By using (i) of Proposition \ref{key-estimate}, the conclusion follows from \cite[Lemma 2.2]{acdh}.
\end{proof}

Recall again that we fix $x_M\in M_0$ and assume $\mathrm{diam}(M_0)=1$.
\begin{prop}\label{large-ball}
Assume that $(D_{N})$ and $(RD_n)$ hold on $M$  with $2<n\le N<\infty$,  and that $(UE)$ and $(P^E_2)$ hold.
Suppose that $(G_{p_0})$ holds for some $p_0\in (2,n)$. Let $\beta>10$ and $2<p<q<p_0$.
There exist $K_0>1$ and $C,C_E>0$ only depending on $n,N,p,q,p_0,\beta$, such that for every
$f\in C^\infty_c(M)$, every $\lambda>0$, $K>K_0$ and $\gamma>0$,
and every ball $B_0=B(x_B,r_B)$, $r_B>\beta$, if there exists $x_0\in B_0$ such that $\mathcal{M}_2(|Tf|)(x_0)\le \lambda$, then it holds
\begin{eqnarray}\label{good-lambda-large-conclusion}
&&\mu\left(\left\{x\in B_0:\, \mathcal{M}_2(|Tf|)(x)>K\lambda,\, \mathcal{M}^\#_{T,A}f(x)\le \gamma \lambda,\, \frac{\|f\|_p}{V(x_M,d(x,x_M)+1)^{1/p}}\le C_E\lambda\right\}\right)\nonumber\\
&&\quad\quad \le C\left({\gamma^2+K^{-q}}\right)\mu(B_0).
\end{eqnarray}
\end{prop}
\begin{proof}Let $J,K>1$ and $\gamma,C_E>0$ to be fixed later.
For $\lambda>0$ let
$$E:=\left\{x\in B_0:\, \mathcal{M}_2(|Tf|)(x)>K\lambda,\, \mathcal{M}^\#_{T,A}f(x)\le \gamma \lambda,\, \frac{\|f\|_p}{V(x_M,d(x,x_M)+1)^{1/p}}\le C_E\lambda\right\}$$
and
$$\Omega:=\left\{x\in B_0:\, \mathcal{M}_2 (|TA_{3{r_B}}f|\chi_{3 B_0})(x)>J\lambda,\, \frac{\|f\|_p}{V(x_M,d(x,x_M)+1)^{1/p}}\le C_E\lambda\right\}.$$

{\bf Claim 1.}  There exists $C>0$ such that
\begin{equation}\label{est-omega}
\mu(\Omega)\le CJ^{-q}\mu(B_0).
\end{equation}

Let us prove the claim. First assume $d(x_B,x_M)<100r_B$. Notice that, if
${\|f\|_p}\ge \mu(B_0)^{1/p}\lambda$, then it follows from the doubling property $(D_{N})$
 that for each $x\in B_0$
 \begin{eqnarray*}
 V(x_M,d(x,x_M)+1)&&\le V(x_M,d(x,x_M)+r_B)\le V(x_B,2d(x_B,x_M)+2r_B)\\
 &&\le C\left(\frac{2d(x_B,x_M)+2r_B}{r_B}\right)^NV(x_B,r_B)\le C\mu(B_0),
  \end{eqnarray*}
and hence,  there exists $c_1>0$ such that for each $x\in B_0$ it holds
$$\lambda\le \|f\|_p\mu(B_0)^{-1/p}\le c_1\frac{\|f\|_p}{V(x_M,d(x,x_M)+1)^{1/p}}.$$
By choosing $C_E<1/c_1$ we see that
\begin{eqnarray}\label{est-omega-1}
\mu(\Omega)=0.
\end{eqnarray}

Suppose now ${\|f\|_p}<\mu(B_0)^{1/p}\lambda$. Recall that the $2$-Hardy-Littlewood maximal operator $\mathcal{M}_2$
is bounded on $L^q(M)$ for all $q>2$. By using (ii) of Proposition \ref{key-estimate} and $(D_{N})$ one has
\begin{eqnarray*}
\left(\fint_{3  B_0}|TA_{3  {r_B}}f|^{q}\,d\mu\right)^{1/q}&&\le  \frac{C_2\|f\|_p}{\mu(3  B_0)^{1/p}}<C\lambda,
\end{eqnarray*}
which together with the $(q,q)$ boundedness of $\mathcal{M}_2$ implies that
\begin{equation}\label{est-omega-2}
\mu(\Omega)\le \frac{1}{(J\lambda)^q}\int_{B_0}\mathcal{M}_2 (|TA_{3{r_B}}f|\chi_{3 B_0})^q\,d\mu\le \frac{C}{(J\lambda)^q}\int_{3  B_0}|TA_{3  {r_B}}f|^{q}\,d\mu\le \frac{C}{J^q}\mu(B_0).
\end{equation}

If $d(x_B,x_M)\ge 100 r_B$, then by using (iii) of Proposition \ref{key-estimate},  one has via the fact  $\mathcal{M}_2(Tf)(x_0)\le \lambda$ that
\begin{eqnarray}\label{est-omega-3}
\left(\fint_{3  B_0}|TA_{3  r_B}f|^{q}\,d\mu\right)^{1/q}&&\le \frac{C_3\|f\|_p}{V(x_M,d(x_B,x_M)+1)^{1/p}}+C_{4}\inf_{y\in   B_0}\mathcal{M}_2(|Tf|)(y)\nonumber\\
&&\le C_4\lambda+\frac{C_3\|f\|_p}{V(x_M,d(x_B,x_M)+1)^{1/p}}.
\end{eqnarray}
If
${\|f\|_p}\ge V(x_M,d(x_B,x_M)+1)^{1/p}\lambda$, then by the fact that $d(x_B,x_M)\ge 100 r_B$, we see that for each $x\in B_0$ it holds
$$\lambda\le \frac{\|f\|_p}{V(x_M,d(x_B,x_M)+1)^{1/p}}\le c_2\frac{\|f\|_p}{V(x_M,d(x,x_M)+1)^{1/p}}.$$
By choosing $C_E<1/c_2$ we find
\begin{eqnarray}\label{est-omega-4}
\mu(\Omega)=0.
\end{eqnarray}
Suppose now ${\|f\|_p}<V(x_M,d(x_B,x_M)+1)^{1/p}\lambda$.
Then by the  $(q,q)$ boundedness of $\mathcal{M}_2$ and \eqref{est-omega-3} we obtain
\begin{eqnarray}\label{est-omega-5}
\mu(\Omega)&&\le \frac{1}{(J\lambda)^q}\int_{B_0}\mathcal{M}_2 (|TA_{3{r_B}}f|\chi_{3 B_0})^q\,d\mu\le \frac{C}{(J\lambda)^q}\int_{3  B_0}|TA_{3  r_B}f|^{q}\,d\mu\nonumber\\
&&\le \frac{C}{(J\lambda)^q} \left(\lambda+\frac{\|f\|_p}{V(x_M,d(x_B,x_M)+1)^{1/p}}\right)^q\mu(3B_0)\nonumber\\
&&\le \frac{C}{J^q}\mu(B_0).
\end{eqnarray}
By choosing $0<C_E<\min\{1/c_1,1/c_2\}$, the above estimates \eqref{est-omega-1}--\eqref{est-omega-5} confirm Claim 1.



Let us estimate the measure of the set $E\setminus \Omega$.
Notice that there exists $c_0>0$ only depending on the measure such that if $c_0K^2>1$ then
\begin{equation*}
\mathcal{M}_2(|Tf|\chi_{3  B_0})(x)>K\lambda,
 \end{equation*}
if $x\in E$. Indeed, since $\mathcal{M}_2(|Tf|)(x)>K\lambda$ for $x\in E$ and $\mathcal{M}_2(|Tf|)(x_0)\le\lambda$, there exists a ball $B=B(z,r)$ such that $x\in B$, $x_0\notin B$ and
$$\int_{B}|Tf|^2\,d\mu>K^2\lambda^2\mu(B),$$
and hence
$$\int_{B(x,2r)}|Tf|^2\,d\mu\ge \int_{B}|Tf|^2\,d\mu>K^2\lambda^2\mu(B)\ge c_0K^2\lambda^2V(x,2r)>\lambda^2V(x,2r).$$
This implies that $r<  r_B$. To see this, let us assume $r\ge r_B$. Then since $x\in E\subset B_0=B(x_B,r_B)$, one has $x_0\in B_0\subset B(x,2r)$.
This together with the above inequality implies that $\mathcal{M}_2(|Tf|)(x_0)>\lambda$, which  contradicts
the assumption $\mathcal{M}_2(|Tf|)(x_0)\le \lambda$.
Therefore it holds $r<  r_B$, $B\subset 3  B_0$, and hence $\mathcal{M}_2(|Tf|\chi_{3  B_0})(x)>K\lambda$.

Therefore, there exists $K_0>0$ large enough, such that for any $K>K_0$ and $x\in E$ it holds
\begin{equation*}
\mathcal{M}_2(|Tf|\chi_{3  B_0})(x)>K\lambda.
 \end{equation*}
By letting $K=J+1>K_0$ we find
\begin{eqnarray*}
&&\mu(E\setminus\Omega)\\
&&\le \mu\left(\left\{x\in B_0:\,
\mathcal{M}_2(|TA_{3  r_B}f|\chi_{3  B_0})(x)\le J\lambda, \, \mathcal{M}_2(|Tf|\chi_{3  B_0})(x)>K\lambda, \,  \mathcal{M}^\#_{T,A}f(x)\le \gamma\lambda\right\}\right)\\
&&\le \mu\left(\left\{x\in B_0:\, \mathcal{M}_2(|T(I-A_{3  r_B})f|\chi_{3  B_0})(x)>(K-J)\lambda, \,  \mathcal{M}^\#_{T,A}f(x)\le \gamma\lambda\right\}\right)\\
&&\le \frac{C}{(K-J)^2\lambda^2}\int_{3  B_0} |T(I-A_{3  r_B})f|^2\,d\mu\le \frac{C\mu(B_0)}{\lambda^2}\inf_{x\in B_0:\,\mathcal{M}^\#_{T,A}f(x)\le \gamma\lambda}\left(\mathcal{M}^\#_{T,A}f(x)\right)^2\\
&&\le C\gamma^2\mu(B_0).
\end{eqnarray*}
This together with the estimate  \eqref{est-omega} with $J=K-1$ for $\Omega$ gives the desired result.
\end{proof}

We next show that $(G_{p_0})$ implies $(R_p)$ for $p<p_0$.
\begin{thm}\label{main-weak-p}
Assume that $(D_{N})$ and $(RD_n)$ hold on $M$  with $2<n\le N<\infty$,  and that $(UE)$ and $(P^E_2)$ hold.
Suppose that  $(G_{p_0})$ for some $p_0\in (2,n)$ hold.
Then $(R_p)$ holds for all $p\in (2,p_0)$.
\end{thm}
\begin{proof}
Let $p\in (2,p_0)$ and fix $q\in (p,p_0)$. Let $f\in C^\infty_c(M)$. Then we have  $ |\nabla \mathcal{L}^{-1/2}f|\in L^2(M)$. For each $\lambda>0$, let
$$E_\lambda:=\left\{x\in M:\, \mathcal{M}_2|\nabla \mathcal{L}^{-1/2}f|(x)>\lambda\right\}. $$
Then $\mu(E_\lambda)<\infty$ for each $\lambda>0$.

By \cite[Chapter III, Theorem 1.3]{cw77}, we can find a sequence of balls $\{B_i\}_i$ with finite overlap property, such that
$E_{\lambda}=\cup_{i}B_i$. Moreover, there exists $C_W>1$ such that
there exists $\tilde x_i\in C_WB_i$, $ \mathcal{M}_2(|\nabla \mathcal{L}^{-1/2}|)(\tilde x_i)\le \lambda$ for each $i$.

Fix a sufficient large $K_0>0$ such that Propositions \ref{small-ball} and \ref{large-ball} hold for any $K>K_0$. Let $\gamma>0$ to be fixed later. Set
$$F_{\gamma\lambda}:=\left\{x\in M:\,\mathcal{M}^\#_{T,A}f(x)> \gamma \lambda\right\}$$
and
$$G_{\lambda}:=\left\{x\in M:\, \frac{\|f\|_p}{V(x_M,d(x,x_M)+1)^{1/p}}>C_E\lambda\right\},$$
where $C_E$ is the constant from Proposition \ref{large-ball}.
By noticing that $E_{K\lambda}\subset E_\lambda$, we find
\begin{eqnarray*}
\mu(E_{K\lambda}\setminus (F_{\gamma\lambda}\cup G_\lambda))&&\le \sum_{B_i:\, r_{B_i}<100}\mu((B_i\cap E_{K\lambda})\setminus F_{\gamma\lambda})+\sum_{B_i:\, r_{B_i}\ge 100}\mu((B_i\cap E_{K\lambda})\setminus (F_{\gamma\lambda}\cup G_\lambda))\\
&&=:I_3+I_4.
\end{eqnarray*}
For each $B_i$ with $r_{B_i}<100$, by applying Proposition \ref{small-ball} to the ball $C_WB_i$, one concludes that
$$\mu((B_i\cap E_{K\lambda})\setminus F_{\gamma\lambda})\le \mu((C_WB_i\cap E_{K\lambda})\setminus F_{\gamma\lambda})\le C\left(\gamma^2+K^{-q}\right)\mu(C_WB_i)\le C\left(\gamma^2+K^{-q}\right)\mu(B_i),$$
and hence by the bounded overlap property of $\{B_i\},$ we obtain
\begin{eqnarray*}
I_3&&\le \sum_{B_i:\, r_{B_i}<100} C\left(\gamma^2+K^{-q}\right)\mu(B_i)\le C\left(\gamma^2+K^{-q}\right)\mu(E_{\lambda}).
\end{eqnarray*}
Meanwhile,  Proposition \ref{large-ball} gives  for each $B_i$ with $r_{B_i}\ge 100$ that
\begin{eqnarray*}
\mu((B_i\cap E_{K\lambda})\setminus (F_{\gamma\lambda}\cup G_\lambda))&&\le \mu((C_WB_i\cap E_{K\lambda})\setminus (F_{\gamma\lambda}\cup G_\lambda)))\le C\left(\gamma^2+K^{-q}\right)\mu(B_i),
\end{eqnarray*}
and hence by applying the bounded  overlap  property of $\{B_i\}$ once more, we obtain
\begin{eqnarray*}
I_4&&\le \sum_{B_i:\, r_{B_i}\ge 100} C\left(\gamma^2+K^{-q}\right)\mu(B_i)\le C\left(\gamma^2+K^{-q}\right)\mu(E_{\lambda}).
\end{eqnarray*}
By the estimates of $I_3$ and $I_4$, we conclude that
\begin{eqnarray}\label{est-e-lambda}
\mu(E_{K\lambda})&&\le  C\left(\gamma^2+K^{-q}\right)\mu(E_{\lambda})+\mu(F_{\gamma\lambda})+\mu(G_\lambda).
\end{eqnarray}

It follows from Lemma \ref{sharp-maximal} that
\begin{eqnarray*}
\mu(F_{\gamma\lambda})\le \frac{1}{(\gamma \lambda)^p}\int_M (\mathcal{M}^\#_{T,A}f)^p\,d\mu\le \frac{C}{(\gamma \lambda)^p}\int_M (\mathcal{M}_2f)^p\,d\mu\le \frac{C\|f\|^p_p}{(\gamma \lambda)^p}.
\end{eqnarray*}

Let us estimate $\mu(G_\lambda)$.
If $V(x_M,1)(C_E\lambda)^p\ge \|f\|_p^p$, then since for any $x\in M$ it holds
$$\frac{\|f\|_p}{V(x_M,d(x,x_M)+1)^{1/p}}\le \frac{\|f\|_p}{V(x_M,1)^{1/p}}\le C_E\lambda,$$
we conclude that $G_\lambda=\emptyset$.

 If $V(x_M,1)(C_E\lambda)^p< \|f\|^p_p$,  then we have
$$G_{\lambda}= \left\{x\in M:\, V(x_M,d(x_M,x)+1)<\left(\frac{\|f\|_p}{C_E\lambda}\right)^{p}\right\},$$
and hence
$$\mu(G_\lambda)\le C\left(\frac{\|f\|_p}{C_E\lambda}\right)^{p}.$$
Inserting the estimates of $\mu(F_{\gamma\lambda})$ and $\mu(G_\lambda)$ into the estimate \eqref{est-e-lambda}, we see that
\begin{eqnarray*}
\mu(E_{K\lambda})&&\le C\left(\gamma^2+K^{-q}\right)\mu(E_{\lambda})+\frac{C\|f\|^p_p}{(\gamma \lambda)^p}+\left(\frac{\|f\|_p}{C_E\lambda}\right)^{p}.
\end{eqnarray*}
This implies that for each $\lambda>0$
\begin{eqnarray*}
(K\lambda)^p\mu(E_{K\lambda})&&\le C\left(\gamma^2+K^{-q}\right)(K\lambda)^p\mu(E_{\lambda})+CK^p\gamma^{-p}\|f\|_p^p+CK^p\|f\|_p^p.
\end{eqnarray*}
By taking $K$ large enough first and then $\gamma$ small enough, we see that
\begin{eqnarray*}
\|\mathcal{M}_2(|\nabla \mathcal{L}^{-1/2}f|)\|_{L^{p,\infty}}^p&&\le \frac 12 \|\mathcal{M}_2(|\nabla \mathcal{L}^{-1/2}f|)\|^p_{L^{p,\infty}}+ C(K,\gamma,p)\|f\|_p^p,
\end{eqnarray*}
which implies the Riesz transform is bounded from $L^p(M)$
to $L^{p,\infty}(M)$ for $p\in (2,p_0)$.

Since the Riesz transform is naturally $L^2$-bounded, we conclude that the Riesz transform is $L^p$-bounded for any $p\in (2,p_0)$ via the Marcinkiewicz interpolation theorem.
\end{proof}

\subsection{Harmonic functions and Riesz transform}
\hskip\parindent
We need the following lemmas to conclude Theorem \ref{main-n} and Theorem \ref{main}.
Recall that $M$ is a complete, non-compact manifold with one or more but finitely many ends.

Let $p\in (2,\infty]$. We say that  the local reverse
$L^p$-H\"older inequality for gradients of harmonic functions holds on  $M$,  if for all $r_0>0$ there exists
$C_H(r_0)>0$  such that, for all balls  $B$ with $r_B<r_0$, and each  $u$ satisfying $\mathcal{L} u=0$ in $3B$, it holds
$$\left(\fint_{B}|\nabla u|^p\,d\mu\right)^{1/p}\le \frac {C_H(r_0)}{r_B} \fint_{2B}|u|\,d\mu.
\leqno (RH_{p,\loc})$$
Notice that, if the constant $C_H(r_0)$ can be taken independent of $r_0$, then $(RH_{p,\loc})$ becomes $(RH_{p})$.

We shall need the mean value property
 for harmonic functions (see \cite[Proposition 2.1]{cjks16} for instance).
\begin{lem}\label{harnack}
Assume that $(D_{N})$ holds on $M$  with $0< N<\infty$,  and that  $(UE)$ holds.
 Then for any $\beta\in (0,1)$ there exists $C>0$ depending on $\beta$ such that if $\mathcal{L} u=0$ in $B(x_0,r)$, then
$$\|u\|_{L^\fz(B(x_0,\beta r))}\le C\fint_{B(x_0,r)}|u|\,d\mu.$$
\end{lem}
\begin{proof}
The case $\beta=1/2$ is a well-known fact as a consequence of Sobolev inequality; see \cite[Proposition 2.1]{cjks16} for instance.
The general case for $\beta\in (0,1)$ follows from a simple covering argument.
\end{proof}

\begin{lem}\label{stable-rhp}
Assume that $(D_{N})$ and $(RD_n)$ hold on $M$  with $2<n\le N<\infty$,  and that $(UE)$ and $(P^E_2)$ hold.
Let $p\in (2,n)$. Then $(RH_p)$ holds if and only if $(RH_{p,\loc})$ and $(RH^E_p)$ hold.
\end{lem}
\begin{proof}
It is obvious that $(RH_p)$ implies $(RH_{p,\loc})$ and $(RH^E_p)$. Let us prove the converse side.

If $3B\cap M_0=\emptyset$, then $(RH_p)$ holds by $(RH^E_p)$,
and if $r_B\le 100$, then $(RH_p)$ holds by $(RH_{p,\loc})$.

Assume now $3B\cap M_0\neq\emptyset$ and $r_B>100$. For any $x\in  B$, we set $$r_x:=\max\{10,\min\{d(x,x_M)/10,\,{r_B}/{10}\}\}.$$

If $d(x,x_M)\le 100$, then $r_x=10$, and by applying $(RH_{p,\loc})$
to the ball $B_x:=B(x,r_x)$, we see that
\begin{eqnarray*}
\left(\fint_{B_x}|\nabla u|^p\,d\mu\right)^{1/p}&&\le \frac{C}{10}\fint_{2B_x}|u|\,d\mu\le \frac{C}{d(x,x_M)+1}\fint_{2B_x}|u|\,d\mu\\
&&\le \frac{C}{d(x,x_M)+1}\|u\|_{L^\infty(2B_x)}\le
 \frac{C}{d(x,x_M)+1}\fint_{2B}|u|\,d\mu,
 \end{eqnarray*}
 where the last inequality follows from the fact $2B_x\subset \frac{6}{5}B\subset 2B$ and Lemma \ref{harnack}.

If $d(x,x_M)>100$, then $r_x=\min\{d(x,x_M)/10,\,{r_B}/{10}\}$. Notice that since $\diam{M_0}=1$, $3B(x,r_x)\cap M_0=\emptyset$.
By applying $(RH^E_{p})$ to $B_x:=B(x,r_x)$,  we conclude via Lemma \ref{harnack} once more that
\begin{eqnarray*}
\left(\fint_{B_x}|\nabla u|^p\,d\mu\right)^{1/p}&&\le \frac{C}{r_x}\fint_{2B_x}|u|\,d\mu\le \frac{C}{r_x}\|u\|_{L^\infty(2B_x)}\le \frac{C}{r_x}\|u\|_{L^\infty(\frac {11}{10}B)} \le  \frac{C}{r_x}\fint_{2B}|u|\,d\mu.
\end{eqnarray*}
Noticing that $d(x,x_M)\le d(x,x_B)+d(x_B,x_M)<4r_B+1<5r_B$, and combining the above two estimates, we conclude that for each $x\in  B$, it holds
\begin{eqnarray}\label{rhp-rp1}
\fint_{B_x}|\nabla u|^p\,d\mu&&\le \frac{C}{[1+d(x,x_M)]^p}\left(\fint_{2B}|u|\,d\mu\right)^p.
\end{eqnarray}
Let $k_0=[\log_2 (5r_B)]$. Since $p<n$, by using $(RD_{n})$ and $(D_N)$, we have
\begin{eqnarray}\label{rhp-rp2}
&&\int_{B}\frac{1}{[1+d(x,x_M)]^{p}}\,d\mu(x) \le \int_{B(x_{M},5r_B)}\frac{1}{[1+d(x,x_M)]^{p}}\,d\mu(x)\nonumber\\
&&\quad\quad\le \sum_{k=0}^{-k_0}\int_{2^{k}B(x_{M},5r_B)\setminus 2^{k-1}B(x_{M},5r_B)}\frac{1}{[1+d(x,x_M)]^{p}}\,d\mu(x)+\int_{2^{-k_0}B(x_{M},5r_B)}\frac{1}{[1+d(x,x_M)]^{p}}\,d\mu(x)\nonumber\\
&&\quad\quad\le \sum_{k=0}^{-k_0}\frac{C}{[1+2^{k}r_B]^{p}}V(x_{M},2^{k}5r_B)+V(x_M,1)\nonumber\\
&&\quad\quad\le \sum_{k=0}^{-k_0}\frac{C}{[1+2^{k}r_B]^{p}} 2^{kn}V(x_{M},5r_B)+C2^{-k_0n}V(x_{M},5r_B)\nonumber\\
&&\quad\quad\le \sum_{k=0}^{-k_0}C2^{k(n-p)}r_B^{-p}V(x_{M},5r_B)+Cr_B^{-n}V(x_{M},5r_B)\nonumber\\
&&\quad\quad\le  C\mu(B)r_B^{-p}.
\end{eqnarray}

For each $x\in B$, and every $y\in B_x=B(x,r_x)$, one has
$$d(x,x_M)-r_x\le d(y,x_M)\le d(x,x_M)+r_x.$$
As a consequence of $d(x,x_M)>100$ and $r_B>100$, it holds $9r_x/10\le r_y\le 11r_x/10$, where
$r_y=\max\{10,\min\{d(y,x_M)/10,\,{r_B}/{10}\}\}$. This together with the doubling condition leads to
\begin{eqnarray*}
\int_B\fint_{B_x}|\nabla u(y)|^p\,d\mu(y)\,d\mu(x)&&=\int_B\int_{B}|\nabla u(y)|^p\frac{\chi_{ B(x,r_x)}(y)}{V(x,r_x)}\,d\mu(y)\,d\mu(x) \\
&&\ge C\int_B\int_{B}|\nabla u(y)|^p\frac{\chi_{ B(x,r_x)}(y)}{V(y,r_y)}\,d\mu(y)\,d\mu(x) \\
&&\ge C\int_B\int_{B}|\nabla u(y)|^p\frac{\chi_{ B(y,10r_y/11)}(x)}{V(y,r_y)}\,d\mu(x)\,d\mu(y) \\
&&\ge C\int_B\nabla u(y)|^p\,d\mu(y).
\end{eqnarray*}
This together with \eqref{rhp-rp1} and \eqref{rhp-rp2} gives that
\begin{eqnarray*}
\int_B |\nabla u(y)|^p\,d\mu(y) &&\le \int_B\fint_{B_x}|\nabla u(y)|^p\,d\mu(y)\,d\mu(x)\le \int_B\frac{C}{[1+d(x,x_M)]^p}\left(\fint_{2B}|u|\,d\mu\right)^p\,d\mu(x)\\
&&\le C \mu(B)r_B^{-p}\left(\fint_{2B}|u|\,d\mu\right)^p,
\end{eqnarray*}
which is nothing but $(RH_p)$.
\end{proof}

\begin{lem}\label{har-pend-ploc}
Assume that $(D_{N})$ holds on $M$  with $0< N<\infty$,  and that $(P^E_2)$ and $(UE)$ hold.
Let $p\in (2,\infty]$. If $(RH^E_p)$ holds on $M$, then $(RH_{p,\loc})$ holds on $M$.
\end{lem}
\begin{proof}
The proof is similar to that of Lemma \ref{pend-ploc},  by using Yau's gradient estimates for harmonic functions.

Recall that Yau's gradient estimate states that if $u$ is a harmonic function on $2B$, then it holds
$$\sup_{x\in B}\frac{|\nabla u(x)|}{u(x)}\le C(N)\left(\frac{1}{r_B}+\sqrt{K}\right),$$
where $K\ge 0$, if every point in $2B$ has Ricci curvature not less than $-K$; see \cite{chy,ya75}.

For any $r_0>0$,  the Ricci curvature  on the set $\{x\in M:\, \dist (x,M_0)<6r_0\}$ is bounded below by a constant $-K(r_0)$ depending on $r_0$, $K(r_0)\ge 0$.
Suppose that $u$ is a harmonic function on $3B$, where $B=B(x,r)$ with $r<r_0$.

If $\dist(x,M_0)\le 3r_0$, then for an arbitrary $\ez>0$, applying the pointwise Yau's gradient estimate to $u+\|u\|_{L^\infty(\frac 32B)}+\ez$, one has for each $y\in B$
\begin{eqnarray*}
|\nabla u(y)|&&\le C\left(u+\|u\|_{L^\infty(\frac 32B)}+\ez\right)\left(\frac{1}{r}+\sqrt{K(r_0)}\right) \le \frac{C(r_0)}{r} \left(\|u\|_{L^\infty(\frac 32B)}+\ez\right).
\end{eqnarray*}
By Lemma \ref{harnack} and letting $\ez\to 0$, we see that
\begin{eqnarray*}
|\nabla u(x)|&&\le \frac{C(r_0)}{r}\fint_{2B}|u|\,d\mu,
\end{eqnarray*}
and hence,
\begin{eqnarray*}
\left(\fint_B |\nabla u|^p\,d\mu\right)^{1/p}&&\le \frac{C(r_0)}{r}\fint_{2B}|u|\,d\mu.
\end{eqnarray*}
If $\dist(x,M_0)> 3r_0$, then $B(x,3r)\cap M_0=\emptyset$ for any $r<r_0$. By using $(RH^E_p)$, one sees that
\begin{eqnarray*}
\left(\fint_B |\nabla u|^p\,d\mu\right)^{1/p}&&\le \frac{C}{r}\fint_{2B}|u|\,d\mu,
\end{eqnarray*}
as desired.
\end{proof}

The above lemma leads to the following open-ended character  of  condition $(RH_p)$ for $p<n$.
\begin{lem}\label{lem-open}
Assume that $(D_{N})$ and $(RD_n)$ hold on $M$  with $2<n\le N<\infty$,  and that $(UE)$ and $(P^E_2)$ hold.
Let $p\in (2,n)$. Then if $(RH_p)$ holds, there exists $\epsilon>0$ such that $p+\epsilon<n$ and $(RH_{p+\epsilon})$ holds.
\end{lem}
\begin{proof}
For each ball $B=B(x_B,r_B)$ with $3B\cap M_0=\emptyset$, we  can find via the $\epsilon$-net argument (cf. \cite[p.102]{hkst})
a sequence of balls $\{B_i\}_{1\le i\le k_0}$, where $1\le k_0\le 8^{N_\mu}$ and $N_\mu>0$ depending only on
the measure $\mu$ (see the proof of Proposition \ref{exp-poincare}),   such that each $B_i$
has radius of $r_B/4$, $\frac 12 B_i\cap \frac 12B_j=\emptyset$ whenever $i\neq j$,
 and $\cup_{1\le i\le k_0}\frac 12B_i\subset B\subset \cup_{1\le i\le k_0} B_i$.

As $B_i$ has radius of $r_B/4$, $\frac 12B_i\subset B$ and $3B\cap M_0=\emptyset$, we find that
$4B_i\subset 2B$ and $4B_i\cap M_0=\emptyset$.

Let $v$ be a harmonic function on $3B$.  Using $(P^E_2)$,  one can conclude from $(RH_p)$ that for each $1\le i\le k_0$, it holds
$$\left(\fint_{B_i}|\nabla v|^p\,d\mu\right)^{1/p}\le \frac{C}{r_B}\left(\fint_{2B_i}|v-v_{2B_i}|^2\,d\mu\right)^{1/2} \le C\left(\fint_{2B_i}|\nabla v|^2\,d\mu\right)^{1/2}.$$
Moreover, this estimate holds for any sub-ball $\tilde B_i$ with $2\tilde B_i\subset 2B_i$, since $4\tilde B_i\cap M_0=\emptyset$.
Applying Gehring's lemma (cf. \cite{ge73}) and self-improvements of the reverse H\"older inequality (cf. \cite[Appendix]{bcf14}), we see that it holds for some $\epsilon>0$,
\begin{equation*}
\left(\fint_{B_i}|\nabla v|^{p+\epsilon}\,d\mu\right)^{1/(p+\epsilon)}\le C\left(\fint_{2B_i}|\nabla v|^2\,d\mu\right)^{1/2}\le
\frac {C }{r_B} \fint_{4B_i}|v|\,d\mu\le \frac {C }{r_B} \fint_{2B}|v|\,d\mu,
\end{equation*}
where the second inequality follows from $(RH_p)$, and hence,
\begin{equation*}
\left(\fint_{B}|\nabla v|^{p+\epsilon}\,d\mu\right)^{1/(p+\epsilon)}\le \sum_{i=1}^{k_0}\frac{\mu(B_i)^{1/(p+\epsilon)}}{\mu(B)^{1/(p+\epsilon)}}\left(\fint_{B_i}|\nabla v|^{p+\epsilon}\,d\mu\right)^{1/(p+\epsilon)}\le \frac {C }{r_B} \fint_{2B}|v|\,d\mu,
\end{equation*}
i.e., $(RH_{p+\epsilon}^E)$ holds. This together with Lemma \ref{stable-rhp} and Lemma \ref{har-pend-ploc} completes the proof.
\end{proof}

We can now finish the proof of Theorem \ref{main-n}.
\begin{proof}[Proof of Theorem \ref{main-n}]
Since in our setting, $(D)$ and $(UE)$ hold,  and $(P_{2,\loc})$ follows from $(P^E_2)$ by Lemma \ref{pend-ploc}, we
see that $(RH_p)\Leftrightarrow (G_p)$ holds for any $p\in (2,\infty)$
by \cite[Theorem 1.5]{cjks16}. The implication
$(R_p)\Rightarrow (G_p)$ holds automatically by the analyticity on $L^p$ of the heat semigroup; see \cite{acdh} for instance.

Finally, if  $(RH_p)$ holds, then by Lemma \ref{lem-open}
there exists $\epsilon>0$ such that $p+\epsilon<n$ and $(RH_{p+\epsilon})$ holds. This implies $(G_{p+\epsilon})$, which by Theorem \ref{main-weak-p} gives $(R_q)$ for all $q\in (2,p+\epsilon)$, in particular, $(R_p)$. The proof is complete.
\end{proof}

Theorem \ref{main} follows from Lemma \ref{stable-rhp} and Theorem \ref{main-n}.
\begin{proof}[Proof of Theorem \ref{main}]
By Theorem \ref{main-n}, $(R_p)$ is equivalent to $(RH_p)$. By Lemma \ref{stable-rhp} and Lemma \ref{har-pend-ploc}
one sees that $(RH_p)$ is equivalent to $(RH_p^E)$.
\end{proof}

\begin{proof}[Proof of Corollary \ref{cor-open}]
This corollary follows from Theorem \ref{main} and Lemma \ref{lem-open}.
\end{proof}

\begin{lem}\label{asy-yau}
Assume that $(D_{N})$ holds on $M$  with $0<N<\infty$, and that $(QD)$ and $(UE)$ hold.
Then $(RH_\infty^E)$ holds.
\end{lem}
\begin{proof}
Suppose that $u$ is a harmonic function on $3B$, with $3B\cap M_0=\emptyset$.
For an arbitrary $\ez>0$, applying the pointwise Yau's gradient estimate (see the proof of Lemma \ref{har-pend-ploc}) to $u+\|u\|_{L^\infty(\frac 32B)}+\ez$, one has for each $x\in B$
\begin{eqnarray*}
|\nabla u(x)|&&\le C\left(u+\|u\|_{L^\infty(\frac 32B)}+\ez\right)\left(\frac{1}{r_B}+\frac{C_M}{r_B+1}\right)
\le \frac{C}{r_B} \left(\|u\|_{L^\infty(\frac 32B)}+\ez\right)\le \frac{C}{r_B}\fint_{2B}(|u|+\epsilon)\,d\mu.
\end{eqnarray*}
Above in the last step we used Lemma \ref{harnack}. Letting $\ez\to 0$, we see that  $(RH^E_\infty)$ holds.
\end{proof}

\begin{proof}[Proof of Corollary \ref{main-1}]
By \cite{cd99} it is known that under $(D)$ and $(UE)$, $(R_p)$ holds for all $p\in (1,2]$.

Notice that under $(QD)$, $(P^E_2)$ holds by Theorem \ref{asy-poin}, and $(RH^E_\infty)$ holds by Lemma \ref{asy-yau}.
Since $(RH^E_\infty)$ implies $(RH_p^E)$  for all $p\in (2,\infty)$, we see that $(R_p)$ holds for all $p\in (2,n)$ by Theorem \ref{main}.
\end{proof}
%

\section{Riesz transform for $p$ above the upper dimension}\label{sec-riesz-har}
\hskip\parindent In this section, we provide the proofs for Theorem \ref{main-one}  and Corollary \ref{main-p2-negative}.
The ideas employed come from recent developments of the elliptic theory for heat kernels from \cite{bcf14,cjks16}.

\begin{proof}[Proof of Theorem \ref{main-one}]
By Theorem \ref{n-poin}, we see that $(P_p)$ holds for any $p>N\vee 2$.
Since our measure is doubling, and $(UE)$ holds, we can apply \cite[Theorem 1.9]{cjks16}  to show that
the conditions $(R_p)$, $(RH_p)$ and $(G_p)$ are equivalent.

Notice that, by \cite{cd99}, $(D)$ and $(UE)$ implies that $(R_q)$ holds for all $q\in (1,2)$.
Since $p>N$, $M$ is $p$-parabolic. By \cite[Theorem C]{ca16} and the assumption that $M$ is non-parabolic,
$(R_p)$ together with $(R_{\frac{p}{p-1}})$ implies that $M$ can have only one end.
\end{proof}
\begin{rem}\rm \label{remark-p2}
Notice that $(P_p)$ together with $(G_p)$ implies $(P_2)$ by \cite{bcf14}. One may also use $(P_2)$ to show that
there exists only one end if $n>2$.
\end{rem}

\begin{proof}[Proof of Corollary \ref{main-p2-negative}]
Suppose that there exists a non-constant harmonic function $u$ on $M$ with the growth
$$u(x)=\mathcal{O}(d(x,o)^\alpha)\ \mbox{as}\, d(x,o)\to\infty$$
for some $\alpha\in [0,1)$ and a fixed $o\in M$.

Assume first $\alpha=0$. If $(R_p)$ holds for some $p>N\vee2$, then by Theorem \ref{main-one}
we have $(RH_p)$, which implies that for all balls $B$ with $r_B>1$ and harmonic function $v$ in $3B$ it holds
\begin{equation}\label{rhc}
\left(\int_{B}|\nabla v|^p\,d\mu\right)^{1/p}\le \frac {C\mu(B)^{1/p}}{r_B} \fint_{2B}|v|\,d\mu\le \frac {CV(x_B,1)^{1/p}r_B^{N/p}}{r_B} \fint_{2B}|v|\,d\mu.
\end{equation}
Applying this estimate to $u$ and letting the radius of $B$ tend to infinity, we see that $\||\nabla u|\|_{p}=0$, which cannot be true.
Therefore the Riesz transform is not  bounded on $L^p(M)$ for any $p>N\vee 2$.

Assume now $\alpha\in (0,1)$. Suppose first that $\frac{N}{1-\alpha}\le 2$. Notice that it implies $N<2$.
Assume $(R_p)$ holds for some $p>2$. Then the estimate \eqref{rhc} holds for $u$, which further implies that
 \begin{equation*}
\left(\int_{B}|\nabla u|^p\,d\mu\right)^{1/p}\le \frac {C\mu(B)^{1/p}}{r_B} \fint_{2B}|u|\,d\mu\le {CV(x_B,1)^{1/p}r_B^{N/p+\alpha-1}}\to 0,
\end{equation*}
as $r_B\to\infty$, since $p>2\ge \frac{N}{1-\alpha}$. This implies $\||\nabla u|\|_p=0$ which contradicts
with $u$ being non-constant. Therefore, $(R_p)$ does not hold for any $p>2$.

For the cases $\frac{N}{1-\alpha}> 2$, we only need to show that the Riesz transform is not  bounded on $L^p(M)$ for $p=\frac{N}{1-\alpha}$.
Suppose this is not the case. By the validity of $(P_p)$ from Theorem \ref{n-poin} and $(R_p)$, we apply
\cite[Corollary 1.10]{cjks16} to find  that there exists $\epsilon>0$ such that
for each  $v$ satisfying $\mathcal{L} v=0$ in $3B$, it holds
$$\left(\fint_{B}|\nabla v|^{p+\epsilon}\,d\mu\right)^{1/(p+\epsilon)}\le
\frac {C}{r_B} \fint_{2B}|v|\,d\mu.$$
This gives
$$\left(\int_{B}|\nabla v|^{p+\epsilon}\,d\mu\right)^{1/(p+\epsilon)}\le \frac {C\mu(B)^{1/(p+\epsilon)}}{r_B} \fint_{2B}|v|\,d\mu.$$
Applying this estimate to $u$ and using $(D_{ N})$, we conclude that
$$\left(\int_{B}|\nabla u|^{p+\epsilon}\,d\mu\right)^{1/(p+\epsilon)}\le  CV(x_B,1)^{1/(p+\epsilon)}r_B^{N/(p+\epsilon)-1+\alpha}\sim r_B^{\frac{N(1-\alpha)}{N+\epsilon(1-\alpha)}-1+\alpha}\to 0,$$
as  $r_B\to \infty$. This contradicts with $u$ being non-constant.  Therefore,
the Riesz transform cannot be  bounded on $L^p(M)$ for $p=\frac{N}{1-\alpha}$. The proof is completed.
\end{proof}

Carron \cite[Theorem D \& Proposition E]{ca16} had provided some sufficient conditions for both boundedness
 and unboundedness of the Riesz transform for $p>N$, under the requirement of quadratic Ricci curvature decay $(QD)$.
As an application of our criteria above, we can relax the requirement of Ricci curvature bound from \cite{ca16} to
$(P^E_2)$ and $(RH^E_\infty)$ (see Lemma \ref{asy-poin} and Lemma \ref{asy-yau}),
and  show that his condition $(HE_\alpha)$ is also necessary, if $n=N$.
\begin{thm}\label{main-p2}
Assume that $(D_{N})$ and $(RD_n)$ hold on $M$  with $1<n=N<\infty$.
Suppose that $(UE)$, $(P^E_2)$ and $(RH^E_\infty)$ hold.
 Let $p\in (N\vee 2,\infty)$. Then the following statements are equivalent.

(i) $(R_p)$ holds;

(ii) $(RH_p)$ holds;

(iii) $(HE_\alpha)$ holds for some $\alpha\in (1-\frac{N}{p},1]$, i.e.,  there exists $C>0$
such that for any ball $B\subset M$ and any harmonic function $u$ on $3B$, it holds for any $x,y\in B$ that
$$|u(x)-u(y)|\le C\left(\frac{d(x,y)}{r_B}\right)^\alpha\fint_{2B}|u|\,d\mu.\leqno(HE_\alpha)$$
\end{thm}
\begin{proof}
The equivalence $(R_p)$ and $(RH_p)$ is a special case of Theorem \ref{main-one}.
Let us show that $(RH_p)\Leftrightarrow(HE_{\alpha})$.

{\bf Step 1. $(HE_\alpha)\Rightarrow (RH_p)$.}
The case $\alpha=1$ is easy, since $(HE_\alpha)$ implies that for any $x\in B$,
$$|\nabla u(x)|\le \limsup_{y:\,d(x,y)\to 0}\frac{|u(x)-u(y)|}{d(x,y)}\le \frac{C}{r_B}\fint_{2B}|u|\,d\mu,$$
which is $(RH_\infty)$.
%

Suppose now $\alpha\in (1-N/p,1)$. Recall that $p>N\vee 2$.  If $3B\cap M_0=\emptyset$, then $(RH_\infty)$ holds on $B$ by $(RH^E_\infty)$.
If $r_B\le 100$, then  by applying  $(RH_{\infty,\loc})$ from Lemma \ref{har-pend-ploc},
one has that for $u$ satisfying $\mathcal{L}u=0$ on $3B$, it holds
$$\|\nabla u\|_{L^\infty(B)}\le \frac{C}{r_B}\fint_{2B}|u|\,d\mu.$$

Let us consider the remaining case: $3B\cap M_0\neq\emptyset$ and $r_B>100$. For any $x\in B$, if $d(x,x_M)\le 10$, then by applying $(RH_{\infty,\loc})$
to the ball $B_x:=B(x,d(x,x_M)+1)$ and $u-u_{2B_x}$ we see that
$$\||\nabla u|\|_{L^\infty(B_x)}\le \frac{C}{d(x,x_M)+1}\fint_{2B_x}\fint_{2B_x}|u(y)-u(z)|\,d\mu(y)\,d\mu(z).$$
Applying $(HE_\alpha)$ we find
$$\||\nabla u|\|_{L^\infty(B_x)}\le \frac{C[1+d(x,x_M)]^\alpha}{[1+d(x,x_M)]r_B^\alpha}\fint_{2B}|u|\,d\mu.$$
For any $x\in B$, if $d(x,x_M)>10$, then by applying $(RH^E_{\infty})$ to $B_x:=B(x,r_x)$, where $r_x=\min\{d(x,x_M)/10,\,{r_B}/{10}\},$
and using $(HE_\alpha)$, we conclude
\begin{eqnarray*}
\||\nabla u|\|_{L^\infty(B_x)}&&\le \frac{C}{r_x}\fint_{2B_x}\fint_{2B_x}|u(y)-u(z)|\,d\mu(y)\,d\mu(z)\le \frac{Cr_x^{\alpha-1}}{r_B^\alpha}\fint_{2B}|u|\,d\mu.
\end{eqnarray*}
Noticing that $d(x,x_M)<5r_B$, and combining the above two estimates, we conclude that for each $x\in B$, it holds
\begin{eqnarray*}
|\nabla u(x)|&&\le \frac{C}{[1+d(x,x_M)]^{1-\alpha}r_B^\alpha}\fint_{2B}|u|\,d\mu.
\end{eqnarray*}
Notice that $p(1-\alpha)<N$. One has via \eqref{rhp-rp2} that
\begin{eqnarray*}
\left(\int_{B}|\nabla u|^p\,d\mu\right)^{1/p}&&\le C\fint_{2B}|u|\,d\mu \left(\int_{B}\frac{1}{[1+d(x,x_M)]^{(1-\alpha)p}r_B^{p\alpha}}\,d\mu(x)\right)^{1/p}\\
&&\le \frac{C}{r_B^\alpha}\fint_{2B}|u|\,d\mu\left(\int_{B(x_{M},5r_B)}\frac{1}{[1+d(x,x_M)]^{(1-\alpha)p}}\,d\mu(x)\right)^{1/p}\\
&&\le C\mu(B)^{1/p}r_B^{-1}\fint_{2B}|u|\,d\mu .
\end{eqnarray*}
That is nothing but $(RH_p)$.

{\bf Step 2. $(RH_p)\Rightarrow (HE_\alpha)$.}

As $(RH^E_\infty)$ and hence $(RH_{\infty,\loc})$ hold, for a ball $B$ satisfying $3B\cap M_0=\emptyset$ or $r_B\le 100$,
one sees that for each  $v$ satisfying $\mathcal{L} v=0$ in $3B$, it holds that
$$|v(x)-v(y)|\le C\frac{d(x,y)}{r_B}\fint_{2B}|v|\,d\mu.$$

Thus we only need to verify $(HE_\alpha)$ for balls $B$ with large radius and $3B\cap M_0\neq\emptyset$.
By Theorem \ref{n-poin}, $(P_{p})$ holds since $p>N\vee 2$.  By using $(RH_p)$ and $(P_p)$,
we apply \cite[Corollary 1.10]{cjks16}  to see that $(RH_{p+\epsilon})$ holds for some $\epsilon>0$.
 Therefore, for each  $v$ satisfying $\mathcal{L} v=0$ in $3B$, it holds
$$\left(\fint_{B}|\nabla v|^{p+\epsilon}\,d\mu\right)^{1/(p+\epsilon)}\le \frac{C}{r_B}\fint_{2B}|v|\,d\mu.$$

Let $x,y\in B$.
If  $d(x,y)\ge r_B/100$, then by Lemma \ref{harnack} one has that
$$|v(x)-v(y)|\le C\left(\frac{d(x,y)}{r_B}\right)^\alpha(|v(x)|+|v(y)|)\le C\left(\frac{d(x,y)}{r_B}\right)^\alpha\fint_{2B}|v|\,d\mu.$$
Suppose that $d(x,y)< r_B/100$.
Using $(P_{p+\epsilon})$, $(D_{N})$, $(RH_{p+\epsilon})$, and a standard telescopic argument (see Remark \ref{rem-telescopic}) gives that
\begin{eqnarray*}
&&|v(x)-v(y)|\\
&&\le \sum_{j=0}^\infty |v_{B(x,2^{-j+1}d(x,y))}-v_{B(x,2^{-j}d(x,y))}| + |v_{B(x,2d(x,y))}-v_{B(y,d(x,y))}|+\sum_{j=1}^\infty |v_{B(y,2^{-j+1}d(x,y))}-v_{B(y,2^{-j}d(x,y))}| \\
&&\le Cd(x,y)\left(\sum_{j=0}^\infty2^{-j+1}\left(\fint_{B(x,2^{-j+1}d(x,y))}|\nabla v|^{p+\epsilon}\,d\mu\right)^{1/(p+\epsilon)}+\left(\fint_{B(x,2d(x,y))}|\nabla v|^{p+\epsilon}\,d\mu\right)^{1/(p+\epsilon)}\right.\\
&&\quad\quad\left.+\sum_{j=1}^\infty2^{-j}\left(\fint_{B(y,2^{-j+1}d(x,y))}|\nabla v|^{p+\epsilon}\,d\mu\right)^{1/(p+\epsilon)}\right)\\
&&\le Cd(x,y)\left\{\sum_{j=0}^\infty 2^{-j}\left(\frac{r_B^N}{2^{-jN}d(x,y)^NV(x,r_B/8)}\right)^{1/(p+\epsilon)}\left(\int_{B(x,r_B/8)}|\nabla v|^{p+\epsilon}\,d\mu\right)^{1/(p+\epsilon)}\right.\\
&&\quad\quad\quad\quad\quad\quad\left.+ \sum_{j=1}^\infty 2^{-j}\left(\frac{r_B^N}{2^{-jN}d(x,y)^NV(y,r_B/8)}\right)^{1/(p+\epsilon)}\left(\int_{B(y,r_B/8)}|\nabla v|^{p+\epsilon}\,d\mu\right)^{1/(p+\epsilon)} \right\}\\
&&\overset{d(x,y)<r_B/100}{\leq}Cd(x,y)\left\{\sum_{j=0}^\infty2^{-j}\left(\frac{r_B^N}{2^{-jN}d(x,y)^NV(x,r_B/8)}\right)^{1/(p+\epsilon)}
\left(\int_{B(x,r_B/4)}|\nabla v|^{p+\epsilon}\,d\mu\right)^{1/(p+\epsilon)} \right\}\\
&&\le Cr_B\left(\frac{d(x,y)}{r_B}\right)^{1-\frac{N}{p+\epsilon}} \left(\fint_{B(x,r_B/4)}|\nabla v|^{p+\epsilon}\,d\mu\right)^{1/(p+\epsilon)}\le C\left(\frac{d(x,y)}{r_B}\right)^{1-\frac{N}{p+\epsilon}}\fint_{B(x,r_B/2)}|v|\,d\mu\\
&&\le C\left(\frac{d(x,y)}{r_B}\right)^{1-\frac{N}{p+\epsilon}}\fint_{2B}|v|\,d\mu.
\end{eqnarray*}
Above in the fourth inequality we used the facts $d(x,y)<r_B/100$, $V(y,r_B/8)\sim V(x,r_B/8)$ and $B(y,r_B/8)\subset B(x,r_B/4)$, and the last inequality holds since $x\in B$, $B(x,r_B/2)\subset 2B$,
and $V(x,r_B/2)\sim V(\frac 12B)\sim V(2B)$.
The above estimate implies $(HE_{\alpha})$ for $\alpha=1-\frac{N}{p+\epsilon}$, and completes the proof.
\end{proof}

\section{Extensions to Dirichlet metric measure spaces}\label{extension-mms}
\hskip\parindent In this section, we discuss extensions of main results to the setting of Dirichlet metric measure spaces.
Since in a non-smooth setting, local Poincar\'e  inequality (see Lemma \ref{pend-ploc}) and  local smoothness of harmonic functions (see Lemma \ref{har-pend-ploc}) do not follow automatically from the assumptions on ends, we need to consider them as additional
assumptions.
However, other assumptions are the same as in the smooth settings.
As the proofs are basically identical to the smooth settings (see \cite{acdh,cjks16}), we will sketch the proofs in the section.

Let $X$ be a locally compact, separable, metrisable, and connected
space equipped with a  Borel measure $\mu$ that is finite on compact sets and strictly positive on non-empty open sets.
Consider a strongly local and regular Dirichlet form $\E$ on $L^2(X,\mu)$ with dense domain $\D\subset L^2(X,\mu)$ (see \cite{fot}   for precise definitions).  According to Beurling and Deny \cite{bd59},  such a form can be written as
$$\E (f, g)= \int_X \,d\Gamma(f,g)$$
for all $f, g\in \D$, where $\Gamma$ is a measure-valued non-negative
 and symmetric bilinear form defined by the formula
$$\int_X\varphi\,d\Gamma(f,g):=\frac12\left[\E (f,\varphi g) + \E (g,\varphi f)-\E (fg,\varphi)\right]$$
for all $f,g\in \D\cap L^\infty(X,\mu)$ and $\varphi\in \D\cap \ccc_0(X).$
Here and in what follows, $\ccc(X)$ denotes the space of
continuous functions  on $X$ and $\ccc_0(X)$ the space of functions
in $\ccc(X)$ with compact support. We shall assume in addition that $\E$  admits a ``{\it carr\'e du champ}",
meaning that $\Gamma(f,g)$ is absolutely continuous with respect to $\mu$,
for all $f, g\in \D$. In what follows, for
simplicity of notation, we  will denote by  $\langle D f,D
g\rangle$  the energy density
$\frac{\,d\Gamma(f,g)}{\,d\mu}$, and by $|D f|$  the square root
of $\frac{\,d\Gamma(f,f)}{\,d\mu}$.

Since $\E$ is strongly local, $\Gamma$ is local and satisfies the
Leibniz rule and the chain rule; see \cite{fot}. Therefore we can
define $\E(f, g)$ and $\Gamma(f, g)$ locally. Denote by $\D_\loc$
the collection of all $f\in L^2_\loc(X)$ for which, for each
relatively compact set $K\subset X$, there exists a function
$h\in\D$ such that $f = h$ almost everywhere on $K$.  The intrinsic
(pseudo-)distance on $X$ associated to $\E$ is then defined by
$$d(x, y):=\sup\left\{f(x)-f(y):\, f\in\D_\loc\cap\ccc(X),\, |D f|\le 1\mbox{ a.e.}\right\}.$$
We always assume  that $d$ is indeed a distance (meaning that for $x\not= y$, $0<d(x,y)<+\infty$)
 and that the topology induced by $d$ is equivalent to the original topology on $X$.
Moreover, we assume that $(X,d)$ is a complete metric space.

Corresponding to such a Dirichlet form $\E$, there exists an operator,  denoted by $\mathscr{L} $,
acting on a dense domain $\mathscr{D}(\mathscr{L})$ in
$L^2(X,\mu)$, $\mathscr{D}(\mathscr{L})\subset \mathscr{D}$, such that for all
$f\in \mathscr{D}(\mathscr{L})$ and each
$g\in \mathscr{D}$,
$$\int_X f(x)\mathscr{L} g(x)\,d\mu(x)=\mathscr{E}(f,g).$$
The opposite of $\mathscr{L} $ is the infinitesimal generator of the heat semigroup $H_t=e^{-t\mathscr{L}}$, $t>0$.

We assume that $X$ is the union of a compact set $X_0$ and some ends $\{E_i\}_{1\le i\le k}$, $k\in\cn$.
We simply adapt all the notions from previous sections with the Laplace-Beltrami operator $\mathcal{L}$ replaced by
$\mathscr{L}$, the Riemannian gradient $\nabla $ replaced by $D$, and $M_0$ replaced by $X_0$;
see \cite{bcf14,cjks16} for more studies in such settings.

The following result generalizes Theorem \ref{main-n} and Theorem \ref{main} to the metric setting.
\begin{thm}\label{main-small-mms}
Assume that the non-compact Dirichlet  metric measure space $(X,d,\mu,\E)$ satisfies
$(D_{N})$ and $(RD_n)$ with $2<n\le N<\infty$.
Suppose that $(UE)$, $(P_{2,\loc})$ and $(P^E_2)$ hold.  Let $p\in (2,n)$. Then the following statements are equivalent.

(i) $(R_p)$ holds;

(ii) $(RH_p)$ holds;

(iii) $(RH^E_p)$ and $(RH_{p,\loc})$ hold;

(iv) $(G_p)$ holds.
\end{thm}
\begin{rem}\rm
Comparing to Theorem \ref{main-n} and Theorem \ref{main}, $(P_{2,\loc})$ is an additional assumption. Notice that, in
the smooth setting, $(P_{2,\loc})$ follows from $(P^E_2)$ as in Lemma \ref{pend-ploc}, however,  in the non-smooth setting, this is not true in general.
Also in the term (iii), we need  $(RH_{p,\loc})$ additionally, since in metric setting, harmonic functions
are not necessarily smooth (see \cite{cjks16}), and $(RH_{p,\loc})$ does not follow from $(RH^E_{p})$, comparing to Lemma \ref{har-pend-ploc}.
For instance, one can glue two Euclidean ends via a smooth part removing a suitable fractal, where the local Poincar\'e inequality and local smoothness of harmonic functions may not hold.
\end{rem}

\begin{proof}[Proof of Theorem \ref{main-small-mms}]
By  \cite[Theorem 1.6]{cjks16}, we have the equivalence of $(RH_p)$ and $(G_p)$.
Moreover, the same proof of Lemma \ref{stable-rhp} works in the metric setting, which implies that $(RH_p)$
is equivalent to $(RH^E_p)$ together with  $(RH_{p,\loc})$, for $p\in (2,n)$.

It remains to show that $(R_p)$ is equivalent to $(G_p)$. It holds automatically  that $(R_p)$ implies $(G_p)$, see \cite{acdh,cjks16} for instance.
On the other hand, the same proof of Theorem \ref{main-weak-p} gives that $(G_p)$ implies $(R_q)$ for any $q\in (2,p)$.
By the same proof of Lemma \ref{lem-open}, one sees that there exists $\ez>0$ such that  $(G_{p+\ez})$  holds, which then implies $(R_p)$,
and completes the proof.
\end{proof}

We have the following metric version of Theorem \ref{main-one}. Recall  that $N\vee 2$ stands for $\max\{N,2\}$.
\begin{thm}\label{main-large-mms}
Assume that the non-compact Dirichlet  metric measure space $(X,d,\mu,\E)$ satisfies  $(D_{N})$ with $0<N<\infty$.
Suppose that $(UE)$, $(P_{2,\loc})$ and $(P^E_2)$ hold. Let $p\in (N\vee 2,\infty)$. Then the following statements are equivalent.

(i) $(R_p)$ holds;

(ii) $(RH_p)$ holds;

(iii) $(G_p)$ holds.
\end{thm}
\begin{proof}
Notice that, by the same proof of Theorem \ref{n-poin},  $(P_{2,\loc})$ and $(P^E_2)$ imply that Poincar\'e inequality $(P_q)$ holds
for any $p\in (N\vee 2,\infty)$. \cite[Theorem 1.9]{cjks16} then gives the desired conclusion.
\end{proof}

\section{Applications}\label{sec-ue}
\hskip\parindent
A key tool in the paper is the Gaussian upper bound of heat kernel, i.e.,
there exist $C,c>0$ such that for all $t>0$ and $x,y\in M$,
$$p_t(x,y)\le
  \frac{C}{V(x,{\sqrt t})}\exp\lf\{-\frac{d^2(x,y)}{ct}\r\}.\leqno(UE)
$$

By \cite{gri92,sal,sal2}, it is well known that, $(D)$ together with
 $(UE)$ is equivalent to a Faber-Krahn inequality, and also equivalent to
 a local Sobolev inequality; see also \cite{bcs15}.
Recent result by Grigor'yan and Saloff-Coste  \cite{gri-sal09,gri-sal16} gives a very useful solution to
the stability of $(UE)$ under gluing operations.
The following result follows from \cite[Theorem 3.5]{gri-sal16}, see also \cite[Corollary 4.6]{gri-sal09}.

\begin{thm}\label{upper-heat}
Let $M$ be a  manifold with finitely many ends $\{E_i\}_{1\le i\le k}$, $k\in\cn$. Suppose that
$M$ satisfies $(D)$.
If for each $i$, there exists a manifold $M_i$ satisfying $(D)$ and $(UE)$, and  a compact subset $K_i\subset M_i$, such that $E_i$ is isometric to $M_i\setminus K_i$, then $(UE)$ holds on $M$.
\end{thm}

Above by ``each $M_i$ satisfies $(UE)$", we mean the heat kernel on $M_i$ satisfies $(UE)$, with nothing to do with the gluing manifold $M$.

From Theorem \ref{upper-heat}, we see that,  if $M$ is obtained by gluing some Riemannian  manifolds with non-negative Ricci curvature,
 simply connected nilpotent Lie groups with polynomial growth as well as conic manifolds, together through a compact manifold smoothly, then $M$ satisfies $(UE)$, since $(UE)$ holds on the aforementioned manifolds; see  \cite{Al92,acdh,hak,cjks16} for instance.

As a consequence, our Theorem \ref{main-n}, Theorem \ref{main} and Theorem \ref{main-one} work, if $M$ is obtained by gluing Riemannian  manifolds with non-negative Ricci curvature,  simply connected nilpotent Lie groups as well as conic manifolds, together through a compact manifold smoothly.

Another class of gluing manifolds to which our result can be applied is the manifold obtained by gluing several
cocompact covering Riemannian manifolds  with  polynomial growth deck transformation group together. Here, a manifold  $\widehat{M}$ has a cocompact covering,
 if there is a finitely generated discrete group $G$ with polynomial volume growth of some order
$D >2$, that acts properly and freely on $\widehat{M}$ by isometries, such that the orbit space
$M_G = \widehat M/G$ is a compact manifold. See \cite{dng04a,cjks16} for instance.

Note also, our results also work on the these settings with the Laplace-Beltrami operator replaced by
any uniformly elliptic operators of divergence form, by Theorem \ref{main-small-mms} and Theorem \ref{main-large-mms}.

Let us finish the proof of Theorem \ref{stability-gluing-main} and Corollary \ref{stability-gluing}.
\begin{proof}[Proof of Theorem \ref{stability-gluing-main}]
Notice that $(P_2^E)$ holds automatically, as each $M_i$ satisfies $(P_2^E)$.
Moreover, $(UE)$ follows from Theorem \ref{upper-heat} since $M_i$ supports $(UE)$.

For each $p\in (2,n)$, since $(R_p)$ implies $(RH_p)$ on each $M_i$ by Theorem \ref{main-n}, we see that  $(RH_p^E)$ holds on $M$ and the
conclusion follows from Theorem \ref{main}.
\end{proof}

\begin{proof}[Proof of Corollary \ref{stability-gluing}]
Since $(D)$ plus $(P_2)$ imply  $(UE)$ and $(P^E_2)$, we see that Theorem \ref{stability-gluing-main} applies.

(i) By \cite[Theorem 0.4]{ac05}, for each $M_i$, there exists $\epsilon_i>0$ such that the Riesz transform is bounded on $L^{2+\epsilon_i}(M_i)$.  This implies that there exists $\epsilon>0$,  possibly smaller  than $\epsilon_i$, $1\le i\le k$,
such that $2+\epsilon<n$ and $(R_{2+\epsilon})$ holds  on each $M_i$. By Theorem \ref{stability-gluing-main} we see that $(R_{2+\epsilon})$ holds.

(ii) Since $(R_p)$ holds on each $M_i$, we apply Theorem \ref{stability-gluing-main} to conclude that $(R_p)$ holds on $M$.
\end{proof}

Consider an $n$-dimensional conic manifold $C(X)$ with compact
basis $X$, $C(X):=\rr^+\times X$, where the metric is given by $dr^2+r^2d_X$.
Let $\lambda_1$ be
the smallest nonzero eigenvalue of the Laplacian on the basis $X$.
By  Li \cite{lh99}, the Riesz transform is bounded on
$L^p(C(X))$ for all $p\in (1,p_0)$ and not bounded for $p\ge p_0$,
where
$$p_0:=n\left(\frac n2-\sqrt{\left(\frac{n-2}{2}\right)^2+\lambda_1}\right)^{-1}$$
if $\lambda_1<n-1$ and $p_0=\infty$ otherwise; see also \cite{acdh}. The following question was also asked in \cite{cch06}.

\begin{quest}[Open Problem 8.1 \cite{cch06}]\label{question-Li}
Is a result similar to H.-Q. Li's valid for smooth manifolds with one conic or asymptotically
conic end? What happens for several conic ends?
\end{quest}
 Guillarmou and Hassell \cite{gh08} had solved the above question, which was recovered by  recent work of Carron \cite{ca16}.
Our result also gives a new proof to the above question.

Let us explain how the proof works.

Notice that the measure satisfies $V(x,r)\sim r^n$ for each $x\in M$ and each $r>0$, where $n\ge 2$.
Suppose that the manifold has at least two conic ends.
 If $n\ge 3$, then Corollary \ref{main-1} applies to show that the Riesz transform is bounded on $L^p(M)$,
for any $p\in (1,n)$, while \cite{cch06} (see Theorem \ref{negative}) already implies the Riesz transform cannot bounded for any $p\ge n$ if $n\ge 3$.
If $n=2$, $(R_p)$ holds for any $p\in (1,2]$ by \cite{cd99}, and is not bounded for any $p>2$
by applying Corollary \ref{main-p2-negative} and using the fact that there exists a non-constant harmonic
function of logarithmic growth (cf. \cite[Section 7]{ca16}).

If the manifold has only one conic end, our Theorem \ref{main-p2} and Corollary \ref{main-p2-negative} apply
since  the Ricci curvature satisfies
$$Ric_M(x)\ge -\frac{C_M}{[d(x,x_M)+1]^2}$$
for some $C_M>0$. The existence of harmonic functions of sub-linear growth, and the elliptic H\"older regularity of harmonic functions
can be found in \cite{cz96} and also \cite[Section 7]{ca16}.

Let $(\widetilde{M}, g_0)$ be a simply connected nilpotent Lie group of dimension $n > 2$ (endowed with a left-invariant metric),
and $\nu$ be the homogenous dimension of $\widetilde M$, i.e. for some $o\in \widetilde M$
$$\nu:=\lim_{R\to\infty}\frac{\log V(o,R)}{\log R}.$$
Notice that $\nu\ge n>2$. Let $(M,g)$ be a manifold obtained by gluing $k > 1$ copies of $(\widetilde{M}, g_0)$.
Carron-Coulhon-Hassell  \cite{cch06} showed that $(R_p)$ does not hold if $p\ge\nu $,
and they asked

\begin{quest}[Open Problem 8.3 \cite{cch06}]\label{question-Lie}
Show that the Riesz transform on $(M,g)$ is bounded on $L^p$ for $p\in (1,\nu)$.
\end{quest}
Carron \cite{ca07} had solved the question. Our Corollary \ref{stability-gluing} also provides a proof, by noticing that  $(D_{\nu})$ and $(RD_\nu)$ hold on $M$,
 and on a Lie group of polynomial growth $(P_{2})$ holds and $(R_p)$ holds for all $p\in (1,\infty)$; see \cite{Al92,cjks16,var88}.

\subsection*{Acknowledgments}
\addcontentsline{toc}{section}{Acknowledgments} \hskip\parindent
The author is indebted to Thierry Coulhon for his deep insights of this question, and for
many helpful and inspiring discussions. He wishes to thank Hongquan Li for helpful discussions,
and to thank Gilles Carron for helpful communications, in particular, for sending the author
the paper \cite{ca16}, which inspires Corollary \ref{main-p2-negative}  and Theorem \ref{main-p2} of this paper.
Last but not least, the author wishes to thank the referee for a very detailed report which improves the quality
of the paper.
The author was partially supported by NNSF of China (11922114 \& 11671039).

\noindent Renjin Jiang \\
\noindent  Center for Applied Mathematics\\
\noindent Tianjin University\\
\noindent  Tianjin 300072\\
\noindent China\\
\noindent {rejiang@tju.edu.cn}

\end{document}